\documentclass[12pt,twoside]{amsart}

  \usepackage{url}
  \usepackage{amssymb,amsmath}
  \usepackage{stmaryrd}
     \usepackage{graphicx}

   \usepackage[a4paper, margin=2.5cm]{geometry}

  \input xypic
  \xyoption{all}

 \usepackage{times}

  \usepackage{amsmath, amsthm, amsfonts}

\newtheorem{theorem}{Theorem}[section]
\newtheorem{lemma}[theorem]{Lemma}
\newtheorem{proposition}[theorem]{Proposition}
\newtheorem{corollary}[theorem]{Corollary}

\theoremstyle{definition}
\newtheorem{definition}[theorem]{Definition}

\theoremstyle{remark}
\newtheorem{remark}[theorem]{Remark}

\numberwithin{equation}{section}

  \newcommand{\cA}{{\mathcal A}}
     \newcommand{\cB}{{\mathcal B}}
  \newcommand{\cE}{{\mathcal E}}
  \newcommand{\cM}{{\mathcal M}}
    \newcommand{\cN}{{\mathcal N}}
  \renewcommand{\cL}{{\mathcal L}}
  \renewcommand{\cD}{{\mathcal D}}

  \newcommand{\cC}{{\mathcal C}}
  \newcommand{\cG}{{\mathcal G }}

    \newcommand{\cV}{{\mathcal  V}}
       
    \newcommand{\cU}{{\mathcal U}}

       \newcommand{\cX}{{\mathcal  X}}

     \newcommand{\cS}{{\mathcal  S}}

   \newcommand{\ba}{\begin{eqnarray}}
   \newcommand{\na}{\end{eqnarray}}
   \newcommand{\ban}{\begin{eqnarray*}}
   \newcommand{\nan}{\end{eqnarray*}}

    \newcommand{\ind}{{\bf  Index }}

    \newcommand{\Vol}{{\text {Vol}}}  
\renewcommand{\vert}{{\text {vert}}}   
\newcommand{\Hess}{{\text {Hess}}}
\newcommand{\Crit}{{\text {Crit}}}
\newcommand{\Cr}{{\text {Crit}}}
\newcommand{\dist}{{\text {dist}}}
\newcommand{\loc}{{\text {loc}}}

\def \mbf{\mathbf}
\def \mc{\mathcal}

  \newcommand{\C}{{\mathbb C}}
  \newcommand{\R}{{\mathbb R}}
  \newcommand{\Z}{{\mathbb Z}}
    \newcommand{\N}{{\mathbb N}}
  \newcommand{\D}{{\mathbb D}}
    \renewcommand{\P}{{\mathbb P}}

   \newcommand{\p}{\partial}

  \newcommand{\fg}{{\mathfrak g}}

 \newcommand{\fj}{{\mathfrak j}}

 \newcommand{\<}{\langle}
  \renewcommand{\>}{\rangle}

  \newcommand{\disp}{\displaystyle}

\begin{document}

\title{$L^2$-Moduli spaces of symplectic vortices on Riemann surfaces with cylindrical ends }

  \author{Bohui Chen}  
  \address{School of mathematics\\
  Sichuan University\\
 Chengdu, China}
  \email{bohui@cs.wisc.edu}  

   \author{Bai-Ling Wang}
  \address{Department of Mathematics\\
  Australian National University\\
  Canberra ACT 0200 \\
  Australia}
  \email{bai-ling.wang@anu.edu.au}

\date{}



   \begin{abstract}Let $(X,\omega)$ be a compact 
 symplectic manifold with  a Hamiltonian action of a compact Lie group $G$ and $\mu: X\to \mathfrak g$ be its moment map.  In this paper, we  study the $L^2$-moduli spaces of symplectic vortices on Riemann surfaces with cylindrical ends. We studied  a circle-valued  action functional whose gradient flow equation  corresponds to  the  symplectic vortex equations on a cylinder $S^1\times \R$.   Assume that $0$ is a regular value of the moment map $\mu$, we show that the functional is of Bott-Morse type and its critical points of the functional  form  twisted sectors  of  the symplectic reduction 
   (the symplecitc orbifold $[\mu^{-1}(0)/G]$). 
   We show that any  gradient flow lines approaches its limit point exponentially fast.   Fredholm theory and compactness  property   are then established  for the $L^2$-Moduli spaces of symplectic vortices on Riemann surfaces with cylindrical ends. 
     \end{abstract}

   \maketitle
  \tableofcontents

\newpage

\section{Introduction and statements of main theorems}
  
 The symplectic vortex equations on a  Riemann surface  $\Sigma$  associated a principal $G$-bundle $P$ and 
 a Hamiltonian $G$-space $(X, \omega)$,  originally discovered by 
 by K. Cieliebak, A. R. Gaio, and D. A. Salamon \cite{CGS}, and independently by I. Mundet i Riera \cite{MR}, is a 
 system of first order partial differential equations
  \ba\label{eq:1}
  \left\{\begin{array}{l}
\bar{\partial}_{J, A} (u) =0\\
*_\Sigma F_A +  \mu(  u )  = 0
\end{array}\right.
 \na
  for a connection $A$ on $P$ and a $G$-equivariant map $u: P\to X$. See Section \ref{2} for an explanation of the
  notations involved. 
 They are  natural generalisations of  the $J$-holomorphic equation  in a symplectic manifold for  $G$ is trivial, and   of  the  well-known Ginzburg-Landau vortices  in a mathematical model of superconductors for  $\Sigma = \C$  and $X=\C^n$ as a Hamiltonian
 $U(1)$-space.  Ginzburg-Landau vortices have been studied both from mathematicians and physicists' viewpoints. They are  two-dimensional solitons, as time-independent solutions  with finite energy to certain  classical field equations  in Abelian Higgs model, see \cite{JaT} for a complete account of Ginzburg-Landau vortices.

 Since the inception  of these  symplectic vortices, there have been steady   developments in the study of the 
 moduli spaces of symplectic vortices and their associated invariants, the so-called Hamitonian Gromov-Witten invariants.  Many  fascinating conjectures have been proposed, for example see  \cite{CGS},  \cite{GaS} and \cite{Zil}.

  As in Gromov-Witten theory, there are several  main technical issues in the definition of invariants from  symplectic vortices: compactification, gluing theory  and regularization for the moduli spaces of  symplectic vortices.   There have been many works focused  on the compactification issue(\cite{CGRS},\cite{MR},\cite{MT},\cite{Zil},\cite{Ott}). 
On the one hand, when $\Sigma$ is closed,  $X$  is symplectically aspherical and  satisfies some convexity  condition,   A. R. Gaio, I. Mundet i Riera  and D. A. Salamon  in \cite{CGRS}  proved compactness of the  moduli space of symplectic vortices with compact support and bounded energy.  On the other hand, when $G= U(1)$  and $X$ is closed, with strong monotone conditions, I. Mundet i Riera   in \cite{MR}  compactified    the  moduli space of  bounded  energy  symplectic vortices  over a fixed closed Riemann surface.  When $G=U(1)$ and $X$ is a general  compact symplectic manifold,  I. Mundet i Riera and G. Tian conpactified 
 the  moduli space of symplectic vortices with  bounded energy   over smooth curves degenerating to nodal curves.    In particular, the  bubbling off   phenomena near nodal points are new. Energy may be lost and
there are gradient flows of the moment map  instead. This is
  not present in the usual Gromov-Witten theory,  and was elegantly  and carefully   presented in \cite{MT}.  Also, there are some studies on special models such as on the  affine vortices (\cite{Zil}).  
 Based on their compactification, Mundet i Riera and Tian  have a long project on defining Hamiltonian GW
   invariants and almost finished(\cite{MT_HGW}). On the other hand,
  Woodward , following Mundet i Riera's approach(\cite{Mundet}), gave an algebraic geometry approach to define gauged Gromov-Witten invariant(\cite{Woo}), and show its relation to Gromov-Witten invariants of $X\sslash G$ via quantum Kirwan morphisms(\cite{Woo}).

In this paper, we  study the moduli spaces of symplectic vortices on a Riemann surface with cylindrical ends.
In particular,  for a  genus $g$ Riemann surface  with $n$-marked points, we will study  the  $L^2$-moduli space of symplectic vortices  on a Riemann surface $\Sigma$ with a cylindrical end metric
near each marked points. Here  the energy of $(A, u)$, defined to be  the
 Yang-Mills-Higgs  energy functional
\ba\label{YMH:energy0}
E(A, u) = \int_\Sigma \dfrac 12 (|d_Au|^2 + |F_A|^2 + |\mu\circ u |^2) \nu_\Sigma,
\na
is finite. It turns out that the Hamiltonian GW type invariants are very sensitive to the volume forms used near
punctured points. For example, readers may be refer to \S\ref{outlook} for further discussion. 

In Section 2, we briefly review the moduli spaces of symplectic vortices on a  closed Riemann surface as 
developed in   \cite{CGS},     \cite{CGRS}  and \cite{MT}.
  In Section 3, we investigate the asymptotic behaviour of symplectic vortices on a half  cylinder $S^1\times \R^{\geq 0}$ with finite energy. For this we adapt the action functional in \cite{Fr1} and \cite{Fr2} to get a circle-valued functional whose $L^2$-gradient flow equation realizes the symplectic vortex equations (\ref{eq:1}) on $S^1\times \R^{\geq 0}$ in temporal gauge. The critical point set of this functional, modulo gauge transformations  can be identified with 
 \[
  \left( \bigsqcup_{g\in G} \big(\mu^{-1}(0)\big)^g  \right)/G, 
 \]
 as a topological space, where the action of $G$ is given by $h\cdot (x, g) = (h\cdot x, hgh^{-1})$. 
 Under the assumption that $0$ is a regular value of  the moment map $\mu$,  so
  the symplectic reduced space  is   a symplectic  orbifold 
  \[
   \cX_0 =[\mu^{-1}(0)/G].
   \]
  Here we use the square bracket to denote the orbifold structure arising the  locally free action of $G$ on $\mu^{-1}(0)$.     Then the critical point set is 
 diffeomorphic to the inertia orbifold of  the  symplectic  orbifold  $\cX_0$, 
 denoted by
 \[
 I\cX_0 = \bigsqcup_{(g)}  \cX_0^{(g)}
 \]
 where $(g)$ runs over the conjugacy class in $G$ with   non-empty fixed points in $\mu^{-1}(0)$. Note that for a non-trivial conjugacy class $(g)$, $\cX_0^{(g)}$ is often called a twisted sector of $\cX_0$, which is
 diffeomorphic to the orbifold arising from the action of $C(g)$ on $ \mu^{-1}(0)^{(g)}$. Here $C(g)$ denotes the centralizer of $g$ in $G$ for a representative $g$ in the conjugacy class $(g)$.

 Throughout this paper, we assume that $0$ is a regular value of  the moment map $\mu$.  Then we show that  this circle-valued  functional is actually of Bott-Morse type. 
 We also establish  an crucial inequality (Proposition \ref{cru:inequ2})  near each critical point. This  inequality 
 enables us to  establish an exponential decay result  for a symplectic vortex  on   $S^1\times \mathbb R^{\geq 0}$ with finite energy, Cf.  Theorem \ref{decay:exp}.
 
 In Section 4, we study the $L^2$-moduli space  $\cN_\Sigma (X, P)$ of symplectic vortices  on a Riemann surface  $\Sigma$  with  $k$-cylindrical ends,  associated to a principal $G$-bundle and a closed Hamiltonian $G$-manifold $(X, \omega)$.   Applying  the asymptotic analysis in Section 3 to the cylindrical end,
we get a continuous  asymptotic limit map (Proposition \ref{asymp:map} and Subsection 4.2)
  \[
    \p_\infty:   \cN_{\Sigma}(X, P) \longrightarrow (\Cr)^k \cong  (I\cX_0)^k.
    \]
The Yang-Mills-Higgs  energy functional takes discrete values on $ \cN_{\Sigma}(X, P)$  depending on  homology classes  in $ H_2^G(X, \Z)$. 

 Fix a homology class $B\in H_2^G(X, \Z)$, denote by  $\cN_{\Sigma } (X, P, B)$  the $L^2$-moduli  space of  symplectic vortices  on a Riemann surface  $\Sigma$   with the topological type defined by $B$.  Then we develop the Fredholm theory  for 
 $\cN_{\Sigma } (X, P, B)$ and  calculate the expected dimension of the $L^2$-moduli  space of  symplectic vortices with prescribed 
 asymptotic behaviours. 

\vspace{2mm}

\noindent{\bf Theorem A}    (Theorem \ref{APS:index-2}) {\em  Let $ \cN_{\Sigma} (X, P, B; \{(g_i)\}_{i=1, \cdots, k})$ be the subset of
   $\cN_{\Sigma} (X, P, B)$ consisting of 
symplectic vortices  $[(A, u)] $ such that 
  \[
  \p_\infty  (A, u) \in  \big( \cX_0^{(g_1)}  \times \cdots \times\cX_0^{(g_k)}  \big)  \subset (\Cr)^k
     \]
Then   $\cN_{\Sigma} (X, P, B; \{(g_i)\}_{i=1, \cdots, k})$ admits an orbifold  Fredhom system with its virtual dimension given by
\[
2\<c_1^G(TX), B\> + 2(n-\dim G) (1-g_{\Sigma}) - 2\sum_{i=1}^k  \iota_{CR}( \cX_0^{(g_i)},  \cX_0)
\]
where $g_{\Sigma}$ is the genus of the Riemann surface $\Sigma$. Here $\iota_{CR}(\cX_0^{(g_i)},  \cX_0)$ is the degree shift as introduced in \cite{CR}. 
  }

\vspace{2mm}

In Section 5, we also establish the compactness  property  for these  $L^2$-moduli  spaces  of  symplectic vortices   on $\Sigma$ with prescribed  asymptotic behaviours.  We show that there are  two  types of  limiting vortices appearing in the compactification.  The first type occurs as   the  bubbling phenomenon of pseudo-holomorphic spheres at  interior points  just as in the Gromov-Witten theory. To describe this type of  limiting vortices, we introduce the usual weighted trees to classify the resulting topological type.   The second type is due to  the sliding-off  of the Yang-Mills-Higgs energy  along the cylindrical ends as happened in the Floer theory.  The combination of these types of convergence sequences is called the  weak chain convergence in instanton Floer theory in \cite{Don1}.  The choice of cylindrical metric on $\Sigma$ is crucial in our study the compactness property in the sense that these  are only  two  types of  limiting vortices appearing in the compactification of the 
 $L^2$-moduli  spaces  of  symplectic vortices on a cylindrical Riemann surface. 

To describe  the  topological types appearing in the compactification, we introduce a notion of
web of stable weighted trees of the   type $(\Sigma;  B)$    consists of a principal tree $\Gamma_0$  with $k$-tails  and a collections of ordered sequence of trees of finite length 
 \[
 \Gamma_i = \bigsqcup_{j=1}^m T_i(j)
 \]
 for each tail $i=1, \cdots, k$.  See Definition \ref{web:tree}  for a  precise definition.
 Let $\cS_{\Sigma;   B}$ be the set of  webs of stable weighted trees of the   type $(\Sigma;   B)$, which is a partially ordered finite set.  
 For each $\Gamma \in \cS_{\Sigma;   B}$, we associate    an $L^2$-moduli   space     $ \cN_\Gamma$   of   symplectic vortices of type $\Gamma$.  Let  $ \cN_\Gamma  ((g_1), \cdots, (g_k)) $ be the corresponding $L^2$-moduli   space  of   symplectic vortices of type $\Gamma$   with prescribed  asymptotic data in 
 \[
 \cX_0^{(g_1)} \times \cdots \times \cX_0^{(g_k)} \subset (\Cr)^k.
 \]
    Then the main theorem of this paper is to show that the  coarse $L^2$-moduli  space of  symplectic vortices   on $\Sigma$  can be compactified into a stratified topological  space whose strata are labelled by a 
   web of stable weighted trees in $\cS_{\Sigma;  B}$.    In the following theorem, we use the notation $|\cN|$ to denote
    the coarse space of an orbifold topological space $\cN$. 
   
\vspace{2mm}

\noindent{\bf Theorem B}    (Theorem \ref{thm:cpt}) {\em 
Let $\Sigma$ be a Riemann surface of genus $g$  with $k$-cylindrical ends. The coarse space  $L^2$-moduli space 
$|\cN_{\Sigma}(X, P, B)| $ can be compactified to a   stratified  topological space 
\[
| \overline{\cN}_{\Sigma}(X, P, B) | = \bigsqcup_{\Gamma \in \cS_{\Sigma;  B}} |\cN_\Gamma |
\]
such that the top stratum is $|\cN_{\Sigma}(X, P, B)|$. 
Moreover, the  coarse moduli space $$ |\cN_{\Sigma} (X, P, B; \{(g_i)\}_{i=1, \cdots, k})|$$ with a specified asymptotic datum 
can be  compactified to a   stratified topological  space 
\[
|\overline{\cN}_{\Sigma}(X, P, B; \{(g_i)\}_{i=1, \cdots, k}) | = \bigsqcup_{\Gamma \in \cS_{\Sigma;  B}} | \cN_\Gamma ( (g_1), \cdots, (g_k))|.
\]}
    \begin{remark}
Note that the evaluation map has its image in $I\mc X_0$, hence our invariants will define on $H^\ast_{CR}(\mc X_0)$. This is different from  the Hamiltonian Gromov-Witten invariants defined earlier, as the invariants are defined on $H^\ast_G(X)$ in \cite{CGRS} and \cite{Woo} .   Hence the invariants we will define is essentially different from the usual HGW invariants. One may refer to \S\ref{outlook} for further discussion. 
\end{remark}
\vspace{2mm}

We remark that  the compactness properties of the moduli spaces of symplectic vortices have been studied earlier in \cite{MT},   \cite{Ott}, \cite{Zil}, \cite{Zil1} and \cite{SV}.  Under the assumption that  $X$ is a K\"ahler Hamiltonian $G$-manifold with semi-free action, the above  compactness theorem has also been obtained by Venugopalan  in \cite{SV} using  a different approach.

\section{Review of symplectic vortices}\label{2} 

In  this section, we review some of basic facts for the symplectic vortices  following 
\cite{CGRS} \cite{MR} and \cite{MT}.  

\subsection{Symplectic vortex equations} \

Let $(X, \omega)$ be  a $2n$-dimensional  symplectic manifold with a compatible almost complex structure $J$ and a Hamiltonian action of a connected compact Lie group $G$
\[
G \times X \longrightarrow X,  \qquad  (g, x) \mapsto gx.
\]
Let $\fg$ be the Lie algebra of $G$ with a $G$-invariant  inner  product $\< \cdot, \cdot\>$.  Recall that an
action of $G$ on $M$ is Hamiltonian if there exists an  equivariant map, called the moment map, 
 
\[
\mu: X\longrightarrow \fg
\]
 satisfying  the defining property  
 \[
 d\mu_\xi  =   \tilde \xi   \lrcorner  \  \omega = \omega(\tilde \xi, \cdot) , \qquad \text{ for any  } \xi \in \fg.
 \]
   Here the function $\mu_\xi$   is given  by   $\mu_\xi(x) = \< \mu(x), \xi\>$,  and 
 $\tilde\xi$ is the vector field on $X$  defined  by the  infinitesimal action 
 of $\xi \in \fg$ on $X$ 
 \[
 (\tilde \xi  f) (x) = \dfrac{d}{d t} f\left (\exp(-t \xi) x\right), \qquad \text{for} \quad  f  \in C^\infty (X),
 \]
 and the symbol $ \lrcorner$  denotes contraction between differential forms and vector fields.
 Note that the moment map is unique up to a shift by an element $\tau \in Z (\fg)$ (the centre Lie subalgebra of $\fg$).
  See Chapter 2 in \cite{GS} for detailed a discussion on the geometry of moment maps.

  Let $P \to \Sigma$ be smooth (principal) $G$-bundle over  a Riemann surface  $(\Sigma, \fj_\Sigma)$ (not necessarily compact and possibly with boundary).    Let  $   g_{\Sigma} $  be  a Riemannian metric   on $\Sigma$ 
  and  $(*_\Sigma, \nu_\Sigma)$ be  the associated  Hodge star operator and  volume form. 
   Denote by $C^\infty_G(P,  X)$ be the space of smooth $G$-equivariant maps $u: P\to X$ and  by $\cA(P)$ the space of connections on $P$ which is an affine space modelled $\Omega^1(\Sigma,  P^{ad})$.  Here $P^{ad} = P\times_G \fg$ is the bundle of Lie algebras
 associated to the   adjoint representation $ad: G \to GL(\fg)$.  
 
 Denote  the associated fiber bundle of $P$ by  \[
\pi:  Y= P \times_G X \longrightarrow 
\Sigma
\]
 the symplectic fiber bundle. Then  a smooth $G$-equivariant map $u:P\to X$ 
yields a section $\tilde u: \Sigma\to Y$. 
Note that any  connection $A$  on $P$ induces  splittings 
 \[
 TP\cong \pi^\ast T\Sigma\oplus T^{\vert} P,\;\;\;
 TY \cong \pi^*T\Sigma \oplus T^{\vert} Y. 
 \]
The
 covariant derivative $d_A\tilde u\in \Omega^1(\Sigma, \tilde u^\ast T^{\vert}Y)$ is derived   from  $du$as follows:
 $$
 d_Au: \pi^\ast T\Sigma\xrightarrow{du} TY\xrightarrow{projection} T^{\vert}Y.
 $$
For simplicity, we denote $d_A\tilde u$ by $d_Au$ as well.

The {\bf  symplectic vortex equations}  on $\Sigma$  are the following first order partial differential  equations for pairs  $( A, u) \in \cA(P) \times C^\infty_G(P,  X)$
\ba\label{sym:vortex}
\left\{\begin{array}{l}
\bar{\partial}_{J, A} (u) =0\\
*_\Sigma F_A +  \mu(  u )  = 0
\end{array}\right.
\na
where     $F_A$  is the curvature of   the connection $A$.  The  almost   complex  structures  $ j_\Sigma$ and $  J $ define an almost complex structure $J_A$  on $Y$.  
The first equation in (\ref{sym:vortex}) implies that $\tilde u$ is a  $J_A$-hololomorphic section.   In term of the 
covariant derivative $d_A u \in \Omega^1(\Sigma, u^* T^\vert Y)$, it is given by
\ba\label{J:curve}
\bar{\partial}_{J, A} (u) =\dfrac 12 \left( d_A u + J \circ d_A u \circ j_\Sigma\right) =0
\na
in $\Omega^{0, 1} (\Sigma, u^* T^\vert Y)$. 
For the second equation in (\ref{sym:vortex}), we remark that  $\mu\circ u$ is a section of $P^{ad} $  and the Hodge star operator defines a map
\[
*_\Sigma :  \Omega^2(\Sigma, P^{ad}) \longrightarrow \Omega^0(\Sigma, P^{ad}). 
 \]
 Using the Riemannian volume $\nu_\Sigma$,  the second equation in (\ref{sym:vortex}) can be written as
 \ba\label{curv:eq}
 F_A +   \mu ( u ) \nu_\Sigma =0.
 \na

 A solution $(A, u)$  to (\ref{sym:vortex}) is called a 
 symplectic vortex on $\Sigma$ associated to a principal $G$-bundle $P$ and a Hamiltonian $G$-space $X$. 
  Two elements $w = (P, A, u) $ and $w' = (P', A', u')$ are called  equivalent iff there is a bundle 
isomorphism
\[
\Phi:  P' \to P
\]
such that
\[
\Phi^*(A, u) = (\Phi^* A, u \circ \Phi) = (A', u').
\]
When $P$ is evident in the context, we will omit $P$ from the notation and simply call  $(A, u)$  for a   symplectic vortex on $\Sigma$.  As  the   symplectic vortex equations (\ref{sym:vortex})  on $\Sigma$ for a fixed $P$ is invariant under the action of  gauge group
$\cG  (P) = Aut(P)$, the moduli space of  symplectic vortices on $\Sigma$ is the set of solutions to (\ref{sym:vortex}) modulo the gauge transformations. We remark that $P$ is an essential part of symplectic vortices,  in particularly in the study of the compactfications of the moduli spaces of vortices.

Given $u: P\to  X$, there is   an equivariant  classifying map $P\to EG$. Together with   $u: P\to X$ , they define   to a continuous
map
\[
u_G: \Sigma \to  X_G :  = EG\times_G X, 
\]
which in turn  determines a degree $2$  equivariant homology class $[u_G]$  in $H_2^G (X, \Z)$ (if $\Sigma$ is closed) .
 Denote by $\widetilde \cM_{\Sigma} (X, B)$ the space of symplectic vortices  on $\Sigma$ associated $(P, X)$ with
a fixed  equivariant homology class in $B \in H_2^G (X, \Z)$, that means,
\[
\widetilde \cM_{\Sigma} (X, B) =\{(A, u)|  [u_G] = B, (A, u) \text{ satisfies the equations (\ref{sym:vortex})} \}.
\]
  The quotient  of $ \widetilde \cM_{\Sigma} (X, B)$ under  the gauge group $\cG(P)$-action  
\[
\cM_\Sigma (X, B) = \widetilde \cM_{\Sigma} (X, B)/  \cG (P)
\]
is called the  moduli space of  symplectic vortices with a fixed homology class  $B$.

A  solution to (\ref{sym:vortex})  with a fixed $B\in  H^G_2(X, \Z)$  is an absolute minimizer (hence, a critical point)  of the
 Yang-Mills-Higgs  energy functional
\ba\label{YMH:energy}
E(A, u) = \int_\Sigma \dfrac 12 (|d_Au|^2 + |F_A|^2 + |\mu\circ u |^2) \nu_\Sigma.
\na
This is due to the  fact  (Proposition 3.1 in \cite{CGS}) that for any  $(A, u) \in \cA(P) \times C^\infty_G(P, X)$, 
 \begin{equation}\label{energyidentity}
 E(A, u) =  \int_\Sigma  \left( |\bar{\partial}_{J, A} (u)|^2 +\dfrac 12 
|*_\Sigma F_A +  \mu(  u )  |^2\right) \nu_\Sigma +\int_{\Sigma} u^\ast\omega-d\<\mu(u),A\>.
 \end{equation}
Here, $u^\ast\omega-d\<\mu(u),A\>$ is a horizontal  and  $G$-equivaraint 2-form on $P$ and  descends to  a 2-form
$\Sigma$, denoted by the same notation. On the other hand,
 $[\omega -\mu] \in  H^2_G(X)$ is the equivariant cohomology class defined by the equivariant  closed 
 2-form $\omega -\mu \in \Omega^2_G(X)$. The pairing $ \<[\omega -\mu], [u_G]\>$ is computed by
 \[
  \<[\omega -\mu], [u_G]\>  = \int_\Sigma \big( (d_A u)^*\omega - \< \mu(u), F_A\> \big)
  =\int_{\Sigma} u^\ast\omega-d\<\mu(u),A\>.
  \]
  Here $d_A u$ is a  horizontal  and  $G$-equivaraint one-form  on $P$ with values in
  $u^*TX$  and  descends  from  a $u^*T^\vert Y$-valued one form on $\Sigma$,   see  Proposition 3.1 in \cite{CGS}.  
 The   Yang-Mills-Higgs energy functional (\ref{YMH:energy})  and the identity  play   vital roles  in the study of  the  moduli space $\cM_\Sigma (X, B)$.
 \begin{remark}
  We remark that \eqref{energyidentity}
 is true for any surface $\Sigma$. In particular, when $(A,u)$ is a symplectic vortex on 
 $\Sigma$, 
 \begin{equation}\label{energyidentity2}
 E(A, u) =\int_{\Sigma} u^\ast\omega-d\<\mu(u),A\>.
 \end{equation}
 This is the crucial identity for us to define the action functional $\mc L$ in Section \ref{3}.
 \end{remark}
   
 \begin{remark} \begin{enumerate}
\item If $G = U(1)$ the unit  circle in $\C$ and $X =\C^n$ with the usual action of $U(1)$by multiplication, then  symplectic vortex equation is a generalisation of the well-studied vortex equations (Cf. \cite{JaT}). In particular, when $\Sigma$ is compact and $X =\C$,   Bradlow and Garcia-Prada showed that 
the moduli space of  vortices on $\Sigma$  with vortex number 
\[
N = \< c_1(P\times_{U(1)}\C), [\Sigma]\>
\]
 is empty if $N>\Vol(\Sigma)/4\pi$, and is 
$
Sym^N (\Sigma),
$ the $N$-th symmetric product of $\Sigma$, if $N>\Vol(\Sigma)/4\pi$. 
\item As observed in \cite{CGS}, the space $\cA (P) \times C^\infty_G(P, X)$ is an infinite dimensional
Fr\'echet manifold with a natural symplectic structure. The action of gauge group $\cG(P)$ is Hamiltonian 
with a moment map
\[
\cA (P) \times C^\infty_G(P, X) \to C^\infty_G(\Sigma , P^{ad})
\]
defined by $(A, u)\mapsto *F_A + \mu(u)$. Hence, the moduli space of symplectic vortices can be thought 
as a symplectic quotient if the space 
\[
\cS=\{(A, u)|  \bar{\partial}_{J, A} (u) =0 \}
\]
is a symplectic submanifold of $\cA (P) \times C^\infty_G(P, X)$. In practice, the space $\cS$ is not a smooth submanifold in general. It still provides a good guiding principle for the development of Hamiltonian Gromov-Witten
theory. See \cite{AB} \cite{BDW} for some applications of this principle in similar contexts. 
\end{enumerate}
 \end{remark} 
 
 When $\Sigma= S^1\times \R$ with the flat metric $(dt)^2 + (d\theta)^2$  and the standard complex structure $\fj(\partial_t) = \partial_\theta$, with respect to a fixed trivialisation of $P$, we can use the temporal gauge 
 \[
 A = d + \xi (\theta, t) d\theta, \qquad \text{for } \xi:   S^1\times \R \to \fg, 
 \]
 to write the  symplectic vortex equations (\ref{sym:vortex}) as 
\ba\label{vortex:cyl}
\left\{\begin{array}{l}
\dfrac{\p u}{\p t} + J \left( \dfrac{\p u}{\p \theta } + \widetilde{\xi } (\theta, t)  (u(x)) \right) = 0\\
 \dfrac{\p \xi}{\p t} +  \mu(  u )  = 0.
\end{array}\right.
\na 
 This is the downward gradient flow equation for a particular  function on $C^\infty (S^1, X \times \fg)$ defined in Section  \ref{3},  where we will  study this function in  more details.

  \subsection{Moduli spaces of symplectic vortices on a  Riemann surface}\
  
  In the study of the moduli space $\cM_{\Sigma}(X, B)$, we need to develop certain Fredholm theory. This requires 
  some Sobolev completion of the space 
  \[
\widetilde   \cB = \cA(P)\times C_{G, B}^\infty(P, X)
  \]
  where $C_{G, B}^\infty(P, X) =\{ u\in C_G^\infty(P, X)  | [u_G] = B\}$.  The  Sobolev embedding theorem in
  dimension 2 leads to the  $W^{1, p}$ Sobolev spaces for $p>2$.  Depending the question at hand regarding Riemann  surface $\Sigma$ being closed, cylindrical or asymptotically Euclidean at infinite, further careful choice of a   suitable Sobolev spaces is needed. Instead, in this subsection, we only review the linearization of the symplectic
  vortex equations and the gauge transformations  on $\widetilde \cB$ with  the Fr\'echet topology as in \cite{CGS}.
  
  Let $\widetilde \cE \to \widetilde   \cB $ be the  $\cG(P)$-equivariant vector bundle whose fiber over $(A, u)$ is given by
  \[
  \widetilde \cE_{(A, u)} = \Omega^{0, 1} (\Sigma, u^*T^\vert Y) \oplus \Omega^0(\Sigma,  P^{ad}).
  \]
  Then the symplectic vortex equations (\ref{sym:vortex}) defines a $\cG(P)$-equivariant section 
  \[
  S(A, u) = \left(\bar{\partial}_{J, A} (u), 
*_\Sigma F_A +  \mu(  u ) \right) 
\]
such that $\widetilde \cM_{\Sigma}(X, B) $ is the zeros of this section.  The vertical differential of this section, 
denoted by $\cD_{A, u}$, 
together with the linearization   $L_{A, u}$ of the gauge transformation  at $(A, u) \in S^{-1}(0)$ give rise to 
the  deformation complex 
\[
\xymatrix{
  \Omega^0(\Sigma, P^{ad}) \ar[r]^{L_{A, u}\qquad \qquad  } & \Omega^1(\Sigma, P^{ad}) \oplus \Omega^0(\Sigma, u^*T^\vert Y) \ar[r]^{\cD_{A, u} }&  \Omega^{0, 1} (\Sigma, u^*T^\vert Y) \oplus \Omega^0(\Sigma, P^{ad})}.
\]
Here $L_{A, u}$ is  given by
$
L_{A, u} (\eta) = (-d_A \eta, \tilde\eta (u)),
$
 and $\cD_{A, u}$
is the  linearization operator  of the   symplectic vortex equations (\ref{sym:vortex}). 

If $\Sigma$ is closed, using   the usual $W^{1, p}$-Sobolev space for $p>2$,  it was shown in \cite{CGRS} that
the operator 
$
\cD_{A, u}  \oplus L_{A, u}^*$ is a Fredholm operator for any $W^{1, p}$-pair  $ (A, u)$ with real index
given by 
\[ 
(n-\dim G) \chi (\Sigma) + 2 \< u^* \left(  c_1(T^\vert Y) \right), [\Sigma]\>.
\]
Equivalently,  by completing  $\widetilde \cE$ and $\widetilde \cB$ under the
$W^{1, p}$ and $L^p$-norms respectively, the triple $(\widetilde \cB,  \widetilde \cE, S)$ defines
a Fredholm system 
\[
(\cB, \cE, S)
\]
after modulo the $W^{2, p}$ gauge transformation group. The  zero set of $S$  is  the moduli space  $\cM_{\Sigma}(X, B)$. 
The central issue in the study of the moduli space   of symplectic vortices is to establish a  virtual fundamental
cycles as in \cite{FO99} or a virtual system as in \cite{CLW3} for a compactified  moduli space $\cM_{\Sigma}(X, B)$.

  \section{Symplectic vortices on a cylinder $S^1\times \R$} \label{3}

  The symplectic vortex equations  (\ref{vortex:cyl}) on $S^1\times \R$ in temporal gauge  suggests that it is a gradient flow equation  for an action  functional on an infinite dimensional space $\cC$. This functional  has been studied in \cite{CGS},  \cite{Fr1} and \cite{Zil1}. After we studied the critical point set and the Hessian of this functional, we 
  establish an inequality (Proposition \ref{cru:inequ2}) for this functional which plays a crucial role in analysing the 
  asymptotic behaviour of an $L^2$ symplectic vortex on $S^1\times [0, \infty)$. 
  This crucial inequality is  applied to show that
a gradient flow line  $\gamma$
with a finite energy condition
\[
E(\gamma) = \int_0^\infty \|\dfrac{\p \gamma (t)}{\p t} \|^2  dt < \infty
\]
has a well-defined limit point, and 
converges exponentially fast to the limit point.   Similar  exponential decay estimates has also been obtained by 
  in  \cite{MT} and \cite{Zil1} using  different methods. 

\subsection{Action functional for  symplectic  vortices} \ 
   Let $P_{S^1}$ be a principal $G$-bundle over $S^1$, and $\cA_{S^1}$ be the space of smooth connections on $P_{S^1}$ which is an affine space over $\Omega^1(S^1, \fg)$.  Since $C^\infty_G(P_{S^1},X)\cong C^\infty(S^1, X_{G})$,  the connected component  of $\cC$ is identified with $\pi_1^G(X)$. For each $c\in \pi_1^G(X)$ we denote the component by
 $\mc C^c$.

   Now
   choosing  a trivialisation $P_{S^1}\to S^1 \times G$  and the standard metric from $S^1 \cong \R/\Z$, we have the identification 
    \[
 \cC=  C^\infty_G(P_{S^1}, X) \times   \cA_{S^1} \cong C^\infty (S^1,  X  \times  \fg).
  \]
  We sometimes  use the same notation to denote both a map in $C^\infty_G(P_{S^1}, X)$  and in $C^\infty (S^1,  X)$ which should be clear in the context.  We remark that the identification 
  of $\cA_{S^1}$ with $C^\infty (S^1,     \fg) $ is with respect to the trivial connection on $P_{S^1}$.

 With respect to the Fr\'echet topology, $\cC$ is a  smooth manifold whose tangent space at $(x, \eta)$ is
  \[
   T_{(x, \eta)} \cC = \Gamma_{C^\infty}(S^1, x^*TX \times \fg),
   \]
   the space of smooth sections of the bundle $x^*TX \times \fg$. 
   Under the identification $\cC =  C^\infty (S^1,  X \times \fg)$,  the full gauge group $LG=C^\infty(S^1, G)$ acts on $\cC$ 
    by
  \ba\label{loop:action}
 g\cdot (x, \eta) = (g x, g^{-1} \dfrac{dg}{d\theta} +  g^{-1} \eta g).
 \na
Here we simply denote by $ g^{-1} \eta g$   the adjoint action $g^{-1}$ on $\eta$.   

  Let $(x_0, \eta_0)$ and $(x_1, \eta_1)$ be in a connected  component of $\cC$ and 
$\gamma$ be  a path 
  \[
   \gamma(t) = (x(t), \eta(t)) :  I = [0, 1] \to \cC 
  \]
  connecting  $ (x_0, \eta_0)$ and $  (x_1, \eta_1)$. Then $\gamma $ determines a
  pair 
  \[
  ( u_\gamma, A_\gamma) \in C_G^\infty (P_{S^1}  \times I, X) \times \cA ( P_{S^1} \times  I ).
  \]
  Define the  energy  functional  for this path $\gamma$ as
  \ba\label{energy:path}
  E (\gamma) = \int_{S^1\times I } \big( (d_{A_\gamma} u_\gamma)^* \omega -     \< \mu(u_\gamma), F_{A_\gamma}\>     \big).
  \na
Note that if the path $\gamma$ satisfies  the symplectic vortex equations (\ref{vortex:cyl}) on $[0, 1]\times S^1$, then 
$E(\gamma)$ agrees with its Yang-Mills-Higgs energy.   Using the coordinate $(\theta, t)$ for $S^1\times I$, we can compute (cf. \eqref{energyidentity2})
  \[
  E (\gamma) = -  \disp{\int_{S^1\times I} (x(t))^*\omega + \int_{S^1} (\<\mu(x_0), \eta_0 \> -  
\<\mu(x_1), \eta_1 \>  )d\theta.}
\]

  \begin{lemma}\label{E_property}
   Under the identification $C_G^\infty (P_{S^1}  \times I, X) \times \cA ( P_{S^1} \times  I ) \cong
  C^\infty ( S^1   \times I, X  \times  \fg)$, the energy function defined in (\ref{energy:path}) enjoys the following properties.
  \begin{enumerate}
  \item For any $g\in LG$, let $g\cdot \gamma$ be the path obtained from the action of $g$, then
\[
 E (\gamma) =  E ( g\cdot \gamma).
 \]
\item If $\gamma_1$ and $\gamma_2$ are homotopic paths relative to  the boundary point
$(x_0, \eta_0)$ and $(x_1, \eta_1)$, then $E (\gamma_1) = E (\gamma_2)$. 
 \end{enumerate} 
  \end{lemma}  
 \begin{proof}
(1) is obvious. We explain (2). The path $\gamma_1\sharp (-\gamma_2)$ defines
a pair $(u,A)$ on a bundle $P$ over $S^1\times S^1$,  then
$$
E(\gamma_1)-E(\gamma_2)=\<[\omega-\mu], [u_G]\>.
$$
Since $\gamma_1\sim \gamma_2$, $[u_G]=0$. Hence $E(\gamma_1)=E(\gamma_2)$.
 \end{proof}
   We now  define a (circle-valued) function 
  on $\cC$ as follows. For each component $\mc C^c$ we
  fix a based point $(x_c,\eta_c)$.  Given a point 
  $(x, \eta) \in  \cC^c$, let $\gamma:[0,1]\to \mc C^c$ be a path connecting 
  $(x_c,\eta_c)$ and $(x,\eta)$. As above, this can be written as a
  pair  
  \[
  (\tilde x, \tilde \eta) \in  C_G^\infty (P_\Sigma, X) \times \cA_{\Sigma},
  \]
  where $\Sigma=[0,1]\times S^1$ and $\cA_{\Sigma}$ is the space of connections on a principal $G$-bundle $P_\Sigma=P_{S^1}\times [0,1]$.  
  Then we define 
  \ba\label{cL:Sigma}
  \cL_{\Sigma}  (\tilde x, \tilde \eta ) =  E(\tilde x,\tilde \eta).  
 \na 
 For a different extension $(\Sigma', \tilde x', \tilde \eta')$,  by the same argument in the proof of Lemma \ref{E_property},  we know that
 \[ \cL_{\Sigma'}  (\tilde x', \tilde \eta' )  -  \cL_{\Sigma}  (\tilde x, \tilde \eta )     
   =\< [\omega-\mu],  [u_G] \>, 
\]
for some $[u_G]\in H^G_2(X,\Z)$ defined by 
$x$ and $x'$. 
 If $( \tilde x \# \tilde x', \tilde \eta \#\tilde \eta') $ is not smooth, we can  choose a  smooth pair which is 
 homotopic to $( \tilde x \# \tilde x', \tilde \eta \#\tilde \eta') $. 
 Then the  topological invariance ensures that the result does not depend on  the choice of the smooth pair. 
Recall that  $\< [\omega-\mu], \cdot\>$  is the homomorphism 
\[
\< [\omega-\mu], \cdot\> :   H^G_2(X, \Z) \longrightarrow \R.
\] 
The image of $\< [\omega-\mu], \cdot\>$ consists of integer multiples of a fixed positive real number $N_{ [\omega-\mu]}$. Hence, modulo $\Z N_{ [\omega-\mu]}$, $\cL_{\Sigma}  (\tilde x, \tilde \eta )  $ descends to a well-defined function 
\ba\label{cL:[0]}
\cL (x, \eta) = \cL_{\Sigma}  (\tilde x, \tilde \eta )  \mod (\Z N_{ [\omega-\mu]}).
\na
We denote by $\cL:  \cC\to \R/ \Z N_{ [\omega-\mu]}$ the resulting circle-valued function.

 Lemma  \ref{E_property} implies that the following action functional on $\cC$ is well-defined.

 \begin{definition} Given a collection of base points  $\{(x_c, \eta_c) | c\in \pi^G_1(X)\}$ for the connected  components $\cC$ labelled by $\pi_1^G(X)$, let $\tilde \cC_{uni}$ be the associated  universal cover of $\cC$ defined  by the homotopy paths. 
 The action functional on $\tilde \cL:  \tilde \cC_{uni} \to \R$ is defined by  \eqref{cL:Sigma}
 for  a homotopy  path from  $(x, \eta) \in \cC$ to the  base point for the connected component.  The induced function
 \[
 \cL: \cC \longrightarrow \R/\Z N_{ [\omega-\mu] }
 \]
 is   called the action functional on $\cC$. 
 \end{definition}
    
  \begin{remark} \label{cover}
  There is a minimal  covering space  of $\cC$, denoted by $\tilde \cC$,  such that  the action  functional $    \cL$   can be lifted to a $\R$-valued function  $\tilde \cL$ on $\tilde \cC$ and  the following diagram commutes
\ba\xymatrix{
 \tilde \cC_{uni} \ar[rr]^{\tilde \cL} \ar[d] &&  \R \ar[d] \\
 \tilde \cC\ar[rr]^{\tilde \cL} \ar[d] &&  \R \ar[d] \\
  \cC\ar[rr]^{\cL} &&  \R/\Z N_{[\omega-\mu]}.  
}
\na
We write an element of $\tilde \cC$ in the fiber over $(x, \eta) \in \cC$ as an equivalent class a path connecting
$(x, \eta)$ to the base point of the connected component. 
\end{remark}

As the covering map $\tilde \cC_{uni} \to  \cC$ is a local diffeomorphism, 
the  differential  and the Hessian operator of $\cL$ can be calculated by the Fr\'echet derivatives of $\tilde \cL$ on 
 $\tilde \cC_{uni} $ or $\tilde \cC$. 
 For this purpose,  we introduce an $L^2$-inner product on  the tangent bundle $T\cC$, that is, for $(v_1, \xi_1), (v_2, \xi_2) 
\in T_{(x, \eta)} \cC,$ 
\ba\label{L^2:inner}
\< (v_1, \xi_1), (v_2, \xi_2)  \>  = \int_{S^1}\left( \omega(v_1, J v_2)  + \<\xi_1, \xi_2\>\right) d \theta.
\na
 
 \begin{proposition} With respect to the $L^2$-inner product,  the $L^2$-gradient of $\cL$ is given by
\ba\label{grad:cC}
\nabla \cL (x, \eta) = \left( J( \dfrac{\p x}{\p \theta} + \tilde \eta_x ),  \mu(x) \right).
\na
Hence, the critical point set is define by the equations
\ba\label{crit:cC}
 \dfrac{\p x}{\p \theta} + \tilde \eta_x =0 , \qquad \mu(x) =0.
\na
 \end{proposition} 
 \begin{proof}
 Let $ (D\cL)_{(x, \eta)} $ be the first Fr\'echet derivative  of $\cL$, that is, for any
$(v, \xi) \in T_{(x, \eta)} \cC$,
\ba\label{cal:grad}
\begin{array}{lll}
(D\cL)_{(x, \eta)} (v, \xi) &=&  \dfrac{\p }{\p  s} \Big|_{s=0}  \cL (\exp_x (sv), \eta+ s\xi)  \\[3mm]
&=&  -\disp{\int_{S^1} \omega (v,  \dfrac{\p x}{\p \theta})  d\theta+ \int_{S^1} \left(  \< d\mu_x(v), \eta\> + \<\mu(x), \xi\>\right)d\theta
}\\[3mm]
&=& \disp{\int_{S^1} \left(  \omega( \dfrac{\p x}{\p \theta} + \tilde \eta_x, v) +  \< \mu(x), \xi\>\right) d\theta } \\[3mm]
&=& \disp{\int_{S^1} \left(  \omega( J( \dfrac{\p x}{\p \theta} + \tilde \eta_x ) , J v) +  \< \mu(x), \xi\>\right) d\theta }\\[3mm]
&=&  \< (J( \dfrac{\p x}{\p \theta} + \tilde \eta_x ),  \mu(x) ), (v, \xi) \>. 
\end{array}
\na
Hence,  $L^2$-gradient of $\cL$ at $(x, \eta)$ is given by (\ref{grad:cC}). 
The proposition is proved.
\end{proof}

\begin{remark}
The   gradient equation  $\nabla \cL (x, \eta) =0$ can be thought as the Euler-Lagrange equations for the action functional $\cL$.  Moreover, the downward  gradient flow equation of $\cL$ on $\cC$ 
\ba\label{grad:flow}
\dfrac{\p}{\p t} \left(x(t), \eta(t) \right)  = -  \left( J( \dfrac{\p x}{\p \theta} + \tilde \eta_x ),  \mu(x) \right)
\na
is exactly the symplectic vortex equation (\ref{vortex:cyl}) on $S^1\times \R$ in temporal gauge.
\end{remark}

  Before we proceed further, let us investigate the gauge invariance of the action functional $\tilde \cL$.

    \begin{lemma}
  The action functional $\tilde \cL$  on $\tilde \cC$ is invariant under the action of $ L_0G$, the connected component of
  $LG$ of the identity.
  \end{lemma}
  
  \begin{proof} We show that $\tilde \cL$ is constant on any  orbit of $L_0G$, equivalently, for any
  path $\gamma (t) $  in $\tilde \cC$  through  $ \gamma(0) =[x, \eta, [\tilde x] ]$ along the $L_0G$-orbit, we need to prove
  \[
  \dfrac{\p}{\p t} \Big|_{t=0} \tilde   \cL (\gamma (t))  =0.
  \]
  We can assume that the tangent vector defined by $\gamma (t) $ is 
 $
  (-\tilde \xi_x,  \dfrac{\p\xi}{\p \theta} + [\eta, \xi])
$
for $\xi \in L\fg =C^\infty (S^1, \fg)$.
  Then the  calculation in (\ref{cal:grad}) implies
  that 
  \[
  \begin{array}{lll}
&&    \dfrac{\p}{\p t} \Big|_{t=0}  \tilde  \cL (\gamma (t)) \\[3mm]
&=&\<  \nabla \tilde \cL (x, \eta),   (-\tilde \xi_x,  \dfrac{\p\xi}{\p \theta} + [\eta, \xi])\>\\[2mm]
&=& \disp{\int_{S^1}} \left( \omega (\dfrac{\p x}{\p\theta} +\tilde\eta_x, -\tilde \xi_x) +\<\mu(x),  \dfrac{\p\xi}{\p \theta} + [\eta, \xi]\>\right) d\theta\\[3mm]
&=&  \disp{\int_{S^1}} \left( \omega (\dfrac{\p x}{\p\theta}, -\tilde\xi_x) -\omega (\tilde\eta_x, \tilde\xi_x) 
+\<\mu(x),  \dfrac{\p\xi}{\p \theta} \>  + \omega (\tilde\eta_x, \tilde\xi_x)  \right) d\theta\\[3mm]
&=& \disp{\int_{S^1}} \left(  \<d\mu_x (\dfrac{\p x}{\p \theta}), \xi\> 
+\<\mu(x),  \dfrac{\p\xi}{\p \theta} \>   \right) d\theta\\[3mm]
&=& \disp{\int_{S^1} }d  \< \mu (x), \xi> =0. 
\end{array}
\]
Here we applied the equality: $\omega (\tilde\eta_x, \tilde\xi_x) = \<\mu(x),   [\eta, \xi]\>$. 
This completes the proof.
  \end{proof}

 Given $(x, \eta) \in \Crit (\cL)$ and $g \in LG$, by property (1) in Lemma \ref{E_property}, $ g\cdot (x, \eta)$ is also a  critical point. That means, 
the critical point set $\Crit(\cL)$  is $LG$-invariant. Note that the based gauge group
\[
\Omega G =\{g\in LG| g(1) = e, \text{the identity element in $G$}
\}
\]
acts on $\cC$ freely. In the next lemma, we provide a description the critical point set modulo the group 
$\Omega G$ on the set theoretical level. For this purpose, we consider $C^\infty (S^1, \fg)$ as the space of
connections on the trivial bundle $S^1\times G$, where we treat $\xi\in C^\infty (S^1, \fg)$
as a $\fg$-valued 1-form  $\xi d\theta$ on $S^1$. Then there is a holonomy map
\[
Hol:  C^\infty (S^1, \fg) \longrightarrow G.
\]
Note that  $Hol:  C^\infty (S^1, \fg) \longrightarrow G$ is the  universal principal $\Omega G$-bundle with $\Omega G$-action on $ C^\infty (S^1, \fg)$ given  by the gauge transformation.   

\begin{lemma}\label{crit:red}  Modulo the based gauge group $\Omega G$, 
the holonomy map $Hol:  \Crit (\cC)/\Omega G    \longrightarrow G$ 
defines a fibration over $G$ whose fiber over $g\in G$ is 
$(\mu^{-1}(0))^g$,  the $g$-fixed point set 
 in $\mu^{-1}(0)$. That is, we have 
 \[
  \Cr (\cL) /\Omega G   =    \bigsqcup_{g\in G} (\mu^{-1}(0))^g. 
 \]
  \end{lemma} 
\begin{proof} 
 Given $(x(\theta), \eta (\theta) ) \in \Cr (\cL)$, then $x(\theta) \in C^\infty (S^1, \mu^{-1}(0))$ and
 \[
 \dot{x} (\theta) = - \tilde \eta_{x(\theta)}.
 \]
 Solving the above  ordinary differential equation over the interval $x: [0, 2\pi] \to X$  with an initial condition $x(0) = p\in \mu^{-1}(0)$, we get a unique   solution.  The  condition of $x$ be a loop in $X$ is that $\eta$ satisfies the condition
 \[
 x(2\pi) = Hol (\eta) \cdot p =p.
 \]
 Hence, we get 
 \[
 \Crit (\cL) \cong  \{
 (p, \eta)|  p\in \mu^{-1}(0),  \eta \in C^\infty(S^1, \fg), Hol (\eta) \cdot p =p
 \}.
 \]
The action of $\Omega G$ on the right hand side is given by the gauge transformation on the second component.  Note that the holonomy map $Hol: C^\infty(S^1, \fg) \to G$ is  a principal $\Omega G$-bundle.  Any $\Omega G$-orbit  at $\eta$ is determined by $Hol (\eta)$.   So we get the first identification, 
 \[
 \Crit(\cL) /\Omega G  \cong  \{(p, g)| p\in \mu^{-1} (0), g\in G_p\}.
 \]
 Now it is easy to see that the holonomy map on $ \{(p, g)| p\in \mu^{-1} (0), g\in G_p\}$ is just the projection to the second factor, whose fiber at $g$ is $(\mu^{-1}(0))^g$. 
     So the lemma is established. 
 \end{proof}

    \begin{remark}\label{rmk:crit}  Set-theoretically, the critical point set $\Crit(\cL)/LG$ can be identified with
    \[
   I [\mu^{-1}(0)/G] \cong  \left(  \mu^{-1}(0)/G \right)\quad  \sqcup \bigsqcup_{(e) \neq (g)\in \cC(G)} (\mu^{-1}(0))^g/C(g),
 \]
  the  inertia groupoid arising from  the  action groupoid   $[ \mu^{-1}(0)/G]  = \mu^{-1}(0) \rtimes G $. 
 Here $\cC(G)$ is the set of conjugacy class in $G$ with a {\em choice }  function $\cC(G) \to G$
 sending $(g)$ to $g\in (g)$, and $C(g)$ is the centraliser of $g$ in $G$. 
     \begin{enumerate}
\item If $G$ acts on $\mu^{-1}(0)$ freely, then $\Crit(\cL) /LG \cong   \mu^{-1}(0)/G $ is   the  symplectic  quotient (also called the reduced space)  of $(X, \omega)$.
 \item If $G$-action on $\mu^{-1}(0)$  is only locally free, then  $\Crit(\cL) /LG$ is the inertia orbifold of  the symplectic orbifold    $\cX_0 = [ \mu^{-1}(0)/G]$.
 \item If $0$ is not a regular value of $\mu$, then  $\mu^{-1}(0)/G$  admits a  symplectic orbifold stratified space, labelled by orbit types (\cite{SL}). 
  \end{enumerate}
    \end{remark}

     This remark suggests that the critical point set $\Crit(\cL)/LG$ can be endowed with a symplectic orbifold structure when $0$  is  a regular value of $\mu$ and the functional is Bott-Morse type in this case. For this, we may investigate the Hessian operator of $\cL$ 
      at $(x, \eta)\in \Crit (\cL)$.    In this paper, we only consider the case that 0 is a regular value of $\mu$. The inertia orbifold of $\cX_0$ is written as
    \[
 I\cX_0 = \bigsqcup_{(g)}  \cX_0^{(g)}
 \]
 where $(g)$ runs over the conjugacy class in $G$ with   non-empty fixed points in $\mu^{-1}(0)$. Note that for a non-trivial conjugacy class $(g)$, $\cX_0^{(g)}$ is often called a twisted sector of $\cX_0$, which is
 diffeomorphic to the orbifold arising from the action of $C(g)$ on $ \mu^{-1}(0)^{g}$ for a representative $g$ in the conjugacy class $(g)$. Here $C(g)$ denotes the centralizer of $g$ in $G$.

\begin{proposition} \label{Hessian} Assume that 0 is a regular value of $\mu$.
Let $(x, \eta)\in \Crit (\cL)$, the Hessian operator  of $\cL$ at $(x, \eta)$
 \[
 \Hess_{(x, \eta)}:  \Gamma_{C^\infty}(S^1, x^*TX \times \fg)   \longrightarrow  \Gamma_{C^\infty}(S^1, x^*TX \times \fg) \]
 is defined by $(v, \xi) \mapsto  \left( J( \nabla_{\frac {\p}{\p\theta}} x^* v  +  \tilde  \xi  + \nabla_{v }\tilde \eta), d\mu_x (v) \right),$
which   is a symmetric operator with respect to the inner product (\ref{L^2:inner}). 
\end{proposition}

\begin{proof} 
 Given $ (v_1, \xi_1), (v_2, \xi_2) \in T_{(x, \eta)} \cC =  \Gamma_{C^\infty}(S^1, x^*TX \times \fg)$,  the Hessian operator 
 \[
\Hess_{(x, \eta)}:  T_{(x, \eta)}^{L^2} (\cC_{1, p}) \to   T_{(x, \eta)}^{L^2} (\cC_{1, p}) 
\]
is  defined by the second Fr\'echet derivative 
\[   \<  (v_1, \xi_1),  \Hess_{(x, \eta)} (v_2, \xi_2)\> = D^2 \cL_{(x, \eta)} ((v_1, \xi_1), (v_2, \xi_2)).
\]
Denote by $\bar v_1$ the parallel transport along a path $\exp_x(sv_1)$. 
  \[\begin{array}{lll}
    &&  D^2 \cL_{(x, \eta)} ((v_1, \xi_1), (v_2, \xi_2)) \\[2mm]
 &=& \dfrac{d}{ds}\Big|_{s=0} \left(( D\cL)_{\exp_x (sv_2), \eta+s\xi_2} (\bar v_1,  \xi_1) \right) \\[3mm]
 &=& \dfrac{d}{ds}\Big|_{s=0} \left(
  \disp{\int_{S^1} \left(  \omega_{\exp_x (sv_2)} ( \dfrac{\p (\exp_x (sv_2) )}{\p \theta} + \widetilde{(\eta+ s\xi_2)}_{\exp_x (sv_2)},  \bar v_1)  +  \< \mu(\exp_x (sv_2)), \xi_1\>\right) d\theta }
  \right).
 \end{array} 
 \]
 Note that $\dfrac{\p \omega_{\exp_x (sv_2)}  }{\p s}\Big|_{s=0} \left( \dfrac{\p x}{\p \theta} + \tilde \eta_x, \cdot\right ) =0$, as $\dfrac{\p x}{\p \theta} + \tilde \eta_x=0$. We can continue the above calculation as follows
 \[
 \begin{array}{lll}
 &=&   \disp{\int_{S^1}  \left(  \omega  ( \nabla_{\frac {\p}{\p\theta}}  x^*v_2, v_1)  d\theta
   + \omega_x ( \nabla_{v_2}\tilde \eta    +    \tilde  \xi_2 , v_1)   d\theta+ 
 \< d\mu_x (v_2), \xi_1\> \right)}\\[3mm]
 &=&  \disp{\int_{S^1}  \left(  \omega  \big( \nabla_{\frac {\p}{\p\theta}}  x^*v_2 +  \tilde  \xi_2   +\nabla_{v_2}\tilde \eta, v_1\big) + 
\< d\mu_x (v_2), \xi_1\>  \right)  d\theta}\\[3mm]
 &=&  \disp{\int_{S^1}  \left(  \omega   \big(\nabla_{\frac {\p}{\p\theta}} x^* v_2 +  \tilde  \xi_2 + \nabla_{v_2}\tilde \eta, v_1\big) + 
\< d\mu_x (v_2), \xi_1\>  \right)  d\theta}.
\end{array}
\]
Here  $\nabla$ is the Levi-Civita covariant derivative associated to the Riemannian metric on $X$. 
  So the  Hessian operator $\Hess_{(x, \eta)}$
is  defined   by
\[
\Hess_{(x, \eta)} (v, \xi) = \left( J( \nabla_{\frac {\p}{\p\theta}} x^* v  +  \tilde  \xi  + \nabla_{v }\tilde \eta), d\mu_x (v) \right).
\]
It is a symmetric operator  due to  the following identity.
\[
 D^2 \cL_{(x, \eta)} ((v_1, \xi_1), (v_2, \xi_2)) =   D^2 \cL_{(x, \eta)} ((v_2, \xi_2), (v_1, \xi_1)),
 \]
 This identity  can be proved   using the following three identities. 
 \begin{enumerate}
\item $\disp{ \int_{S^1}  \left(   \omega  (\nabla_{\frac {\p}{\p\theta}} x^* v_1,  v_2) 
-  \omega  (\nabla_{\frac {\p}{\p\theta}} x^* v_2, v_1)   \right) d\theta}=  0.$ 
  \item $  \disp{ \int_{S^1}  \left(  
 \omega   \big(   \tilde  \xi_2 , v_1\big) +  \< d\mu_x (v_2), \xi_1\>  \right)  d\theta} 
=  \disp{ \int_{S^1}  \left(  
 \omega   \big(   \tilde  \xi_2 , v_1\big) +   \big(   \tilde  \xi_1 , v_2\big)   \right)  d\theta, }
$
which is symmetric in $(v_1, \xi_1)$ and $ (v_2, \xi_2))$.
\item $ \disp{ \int_{S^1}  \omega  (  \nabla_{v_2}\tilde \eta, v_1)  d\theta} = 
\disp{ \int_{S^1}  \omega  (  \nabla_{v_1}\tilde \eta, v_2)}$. 
\end{enumerate}
\end{proof} 
  
 We can choose a representative for any critical point in $\Crit(\cL)$ according to its holonomy. If a critical
 point has a trivial holonomy, then using a based gauge transformation,  it is gauge equivalent to
 a critical point of the form
 \[
 (x, 0) \in \mu^{-1}(0) \times \fg.
 \]
If a  critical  point has a non-trivial holonomy $g= \exp(2\pi \eta)$ for $\eta\in \fg$, then it is 
gauge equivalent to
 a critical point of the form
 \[
 ( \exp(2\pi t \eta )x, \eta d\theta)
 \]
for $t\in [0, 2\pi]$ and  $x\in (\mu^{-1}(0))^g$. As the Hessian operator is equivariant under the gauge transformation, it is often
simpler to study the Hessian operator at critical points of the above special form.

  \begin{corollary}   The Hessian operator at  the critical point $(x, 0) \in \mu^{-1}(0) \times \fg$ is given by 
\[
\Hess_{(x, 0)} (v, \xi)=   \left(  J( \dfrac{\p v }{\p \theta} +   \tilde  \xi_x ), d\mu_x (v) \right).
\] 
The  Hessian operator  at the critical point of the form $ (x, \eta) =  ( \exp(2\pi\theta \eta) x_0, \eta)
$
for  $\theta \in [0, 2\pi]$ and  $x_0\in (\mu^{-1}(0))^g$  and $g=\exp(2\pi \eta)$  is given by 
\[
\Hess_{( x, \eta)}(v, \xi) \mapsto  \left(  J( -L_{\tilde \eta} v +   \tilde  \xi_x ), d\mu_x (v) \right).
\]
Here $L_{\tilde \eta} v$ is the Lie derivative of $v$ along the vector field $\tilde \eta$. 
 \end{corollary}
  
  \begin{proof}
 It is straightforward  to check that  the Hessian operator at a critical point of the form
 $(x, 0) \in \mu^{-1}(0) \times \fg$ is given by 
$(v, \xi) \mapsto  \left(  J( \dfrac{\p v }{\p \theta} +   \tilde  \xi_x ), d\mu_x (v) \right)$. 

At the critical point of the form $  (x, \eta) =  ( \exp(2\pi\theta \eta) x_0, \eta)$, the vector field $\frac{\p}{d \theta}$  along the loop
$x= \exp(2\pi \theta \eta) x_0$ agrees with $-\tilde \eta$, then we get
\[
\nabla_{\frac {\p}{\p\theta}} x^* v  + \nabla_{v }\tilde \eta  =  -  \nabla_{ \tilde \eta}   v  + \nabla_{v }\tilde \eta  
=   -L_{\tilde \eta} v.
\]
Hence, the Hessian operator at this critical point is given by $(v, \xi) \mapsto  \left(  J( -L_{\tilde \eta} v +   \tilde  \xi_x ), d\mu_x (v) \right)$. 
  \end{proof}
     
     Now we introduce the standard   Banach completion of $\cC$. This Banach set-up is also crucial for the  Fredholm analysis of  the gradient flowlines of $\cL$,   equivalently, the symplectic vortices on $S^1\times \R$. 
 
Consider  the Banach manifold 
 \[
 \cC_{1, p} = \{(x, \eta) \in W^{1, p}(S^1, X\times \fg) \}.
 \]
 Here $p \geq 2$, so $(x, \eta)$ is a continuous map.  For simplicity, one could just take $p=2$.  The tangent space of $\cC_{1, p} $ at $(x, \eta) $ is
 \[
 T_{(x, \eta)}\cC_{1, p} = \Gamma_{W^{1, p}} (S^1, x^* TX  \times \fg),
 \]
 consisting of $W^{1, p}$-sections.  The gauge group for this Banach manifold is 
 the $W^{2, p}$-loop group
 \[
 \cG_{2, p} =W^{2, p}(S^1, G)
 \]
 acting on $\cC_{1, p}$ in the way as  in (\ref{loop:action})  for the smooth case.  Denote
 by $ \cG_{2, p}^0$ the based $W^{2, p}$-loop group. Then the action of $ \cG_{2, p}^0$ on $\cC_{1, p}$ is
 free. 
  
  By the Sobolev embedding theorem, $T_{(x, \eta)}\cC_{1, p}$  is contained in
 the $L^2$-tangent space
 \[
 T^{L^2}_{(x, \eta)}\cC_{1, p} = \Gamma_{L^2} (S^1, x^* TX  \times \fg), 
  \]
  the space of $L^2$-section of the bundle $x^* TX  \times \fg$ on which the 
 $L^2$-inner product (\ref{L^2:inner}) is well-defined and the  $L^2$-gradient $\nabla \cL$ is a
 $L^2$-tangent vector field on $\cC_{1, p}$. 
  Modulo $W^{2, p}$ gauge transformation, the equations (\ref{crit:cC}) is a   first order elliptic equation. By the standard elliptic regularity, we know that modulo gauge transformation, 
 the critical point set $\Crit (\cL)$ consists of smooth loops in $\cC_{1, p}$.  By the same argument,  a solution to the 
 $L^2$ gradient flow equation (\ref{grad:flow}) of $\cL$ on $\cC_{1, p}$  for
 \[
 (x(t), \eta(t)):  [a, b] \longrightarrow \cC_{1, p}
 \]
 with a smooth initial boundary condition
  is gauge equivalent to a smooth 
 symplectic vortex on $S^1\times [a, b]$ in temporal gauge. 
 In this sense, we say that the $L^2$-gradient 
 of $\cL$ and the $L^2$-gradient flow lines are well-defined on  $\cC_{1, p}$. 
 
 Now we explain that  the functional $\cL$ satisfies  certain properties which are analogous to the Morse-Bott 
 properties in  the finite dimension.  
 
 \begin{proposition}\label{Hess:spec}
  Assume that $0$ is a regular value of the moment map $\mu$.  Let $(x, \eta)$ be a critical point of $\Crit (\cL)$, then the Hessian operator  $\Hess_{(x, \eta)}$ of $\cL$ at $(x, \eta)$
 \[
 \Hess_{(x, \eta)}:  T^{L^2}_{(x, \eta)}\cC_{1, p}  \longrightarrow T^{L^2}_{(x, \eta)}\cC_{1, p}
 \]   is an unbounded essentially self-adjoint
operator whose spectrum is real, discrete and unbounded in both directions. 
Moreover,  each non-zero eigenvalue has finite multiplicity, and the tangent space of the $\cG_{2, p}$-orbit through 
$(x, \eta)$, 
$
T_{(x, \eta)} (\cG_{2, p} (x, \eta)) $ is contained in $Ker (Hess_{(x, \eta)}),
$
and its $L^2$-orthogonal is finite dimensional. 
\end{proposition}
\begin{proof}   Note that $\Crit (\cL)$ is invariant under the gauge group $\cG_{2, p}$ and 
the Hessian operator  is  $\cG_{2, p}$-equivariant.   As $0$ is a regular value of $\mu$  and $X$ is compact, there are only finitely many elements of finite order in $G$ with non-empty fixed points in $\mu^{-1}(0)$. By  Lemma \ref{crit:red} and Remark \ref{rmk:crit}, we know that any critical point   is gauge equivalent to
a critical point of the form
\[
(x (\theta), \eta)= (\exp(\theta \eta ) x_0, \eta )  
\]
for $x_0 \in \mu^{-1}(0)$ and  $\eta  \in \fg$ such that $\exp (2\pi \eta ) \in G_{x_0}$ (a finite group).   Therefore, we only need to establish the lemma for 
 $(x(\theta), \eta ) = (\exp(\theta \eta ) x_0, \eta ) \in \Crit (\cL)$ for $\eta \in \fg$. Denote $g=  \exp (2\pi \eta )$.   Note that $ (\exp(\theta \eta ) x_0, \eta )$ is contained in a component of  $\Crit(\cL)$,   whose quotient under the based gauge group $\cG_{2, p}^0$ is diffeomorphic to 
 \[
 (\mu^{-1}(0))^g, 
 \]
under the identification in Lemma \ref{crit:red}.
 
 In this case,  we first  prove  the following  $L^2$-orthogonal decomposition 
 \ba\label{ker:=}
  Ker (Hess_{(x, \eta )}) \cong  T_{(x, \eta )} (\cG^0_{2, p} \cdot (x, \eta ) ) \oplus  T_{x(0)}\big( \mu^{-1} (0)\big)^g .
  \na
 Here $ T_{x(0)}(\mu^{-1} (0))$ is thought as a subspace of $ Ker (Hess_{(x, \eta )}) $ under the identification  in Lemma \ref{crit:red}.  To check (\ref{ker:=}), let 
  $(v, \xi) \in  Ker (Hess_{(x, \eta )})$ be $L^2$-orthogonal to $ T_{(x, \eta_0)} (\cG_{2, p} \cdot (x, \eta ) )$. Then
  $(v, \xi)$ satisfies the following  three equations.
  
  \begin{enumerate}
\item $\disp{\int_{S^1}  }\left(\omega ( -\tilde\zeta_x, v) + \<  \dfrac{\p \zeta}{\p \theta}+ [\eta , \zeta], \xi\>\right) d\theta =0$ for any $\zeta \in W^{2, p}(S^1, \fg)$.\\[3mm]
\item $-L_{\tilde \eta} v+ \tilde \xi_{x(\theta)} =0$.  \\
\item $d\mu_{x(\theta)}(v) =0 \Rightarrow v(\theta) \in   T_{x(\theta)}(\mu^{-1} (0))$.
\end{enumerate}
Due to    the identity $\omega (\tilde \zeta_x, v) = \<d\mu_x (v), \zeta\> =0$  for  $v (\theta) \in  T_{x(\theta)}(\mu^{-1} (0))$ and $\zeta \in W^{2, p}(S^1,  \fg)$, 
 the first equation gives rise to 
\[
\dfrac{\p \xi}{\p \theta}+ [\eta, \xi]=0.\]
This equation  admits a  periodic solution $\xi$  if and only if $ [\eta, \xi] =0$. Hence, $\xi$ is a constant function taking value in the Lie algebra 
\[
\{\xi\in \fg| [\eta, \xi] =0 \}
\]
  of the centraliser  of $g=  \exp (2\pi \eta ) $ in $G$. 
 This implies that $\tilde \xi_{x(\theta)}\in T_{x(\theta)}\big( \mu^{-1} (0)\big)^g$. Then any solution to the second equation is uniquely determined
by an initial value $v(0) \in T_{x(0)}\big( \mu^{-1} (0)\big)^g$.  
 
The subspace of $T_{(x, \eta_0)} (\cC_{1, p} )$ which is  $L^2$-orthogonal   to  $T_{(x, \eta)} (\cG_{2, p} \cdot (x, \eta))$ is given by
\[
\{(v, \xi) | d\mu_x (v) + \dfrac{\p v}{\p \theta} + \tilde \xi_{x}= 0\}, 
\]
on which the Hessian operator is a compact self-adjoint perturbation of a first order elliptic operator 
on $S^1$. The remaining claims in the lemma  follows from the standard elliptic theory on   compact manifolds. \end{proof}

Denote by $\cC^\#_{1, p}$ the submanifold  of $\cC_{1, p}$ consisting of elements with finite stabilisers under the gauge group $\cG_{2, p}$.  Then 
\[
\cB_{1, p}^\# = \cC^\#_{1, p}/\cG_{2, p}
\]
is a smooth Banach orbifold. Let $(x, \eta) \in  \cC^\#_{1, p}$ and let  
\[
\cG_{(x, \eta)} =\{g\in \cG_{2, p}| g\cdot (x, \eta) = (x, \eta) \}
\]
be the stabiliser group of $(x, \eta)$, a finite group in $\cG_{2, p}$.  Then the  tangent space at $\gamma =[x, \eta]  \in \cB_{1, p}^\#$ in orbifold sense  is  a  $\cG_{(x, \eta)}$-invariant  Banach space
\[
 \{ (v, \xi) \in T_{(x, \eta)} \cC_{1, p} | (v, \xi) \text{ is $L^2$-orthogonal to } T_{(x, \eta)}\big( \cG_{2, p} (x, \eta)\big).
 \}
 \]
  The action function $\cL$ descends locally to a circle-valued function on
the Banach orbifold $\cB_{1, p}^\#$.  The $L^2$-gradient vector field  $\nabla \cL$ defines  an orbifold
$L^2$-gradient vector field on $\cB_{1, p}^\#$.  The following corollary follows from
Lemma \ref{crit:red} and  Proposition \ref{Hess:spec}.

\begin{corollary} Assume that  $0$ is a regular value of the moment map $\mu$, then 
 the critical point set 
 \[
 \Cr =\Crit (\cL)/\cG_{2, p} \subset  \cB_{1, p}^\# 
\]
is a smooth orbifold, diffeomorphic to the inertia orbifold of the symplectic reduction 
\[
\cX_0 = [\mu^{-1}(0)/G].
\]
 Each component (called a twisted sector)
is a finite dimensional suborbifold of $\cB_{1, p}^\#$.  
\end{corollary}

 From now on, we assume that  $0$ is a regular value of the moment map $\mu$.   Let  $[x, \eta]\in \Crit (\cL)/\cG_{2, p} $.
  Then  the Hessian operator at $[x, \eta]\in \Crit (\cL)/\cG_{2, p} $  is 
an unbounded essentially self-adjoint Fredholm operator in the orbifold sense
\[
\Hess_{[x, \eta]}:  T^{orb}_{[x, \eta]} \cB_{1, p}^\#  \longrightarrow T^{orb}_{[x, \eta]} \cB_{1, p}^\#. 
\]
with discrete  real spectrum (unbounded in both directions)  of finite multiplicity.  The kernel of
$\Hess_{[x, \eta]}$ is the orbifold tangent space of the critical orbifold at $[x, \eta]$ and the normal
Hessian operator is nondegerate. There is a uniformly lower bound on the absolute value of the non-zero eigebvalues of $\Hess_{[x, \eta]}$ for $[x, \eta]\in \Crit (\cL)/\cG_{2, p} $ as $ \Crit (\cL)/\cG_{2, p} $ is compact. 
 We remark  that the circle valued function $\cL$ on $\cB_{1, p}^\#$ defines a
closed 1-form on $\cB_{1, p}^\#$ whose critical point set is of Morse-Bott type.

In the next lemma, we establish the  inequality  for $\tilde \cL$  which is important in analysing gradient flow lines near any critical point.

\begin{proposition}\label{cru:inequ2}
For any $x$ in a  critical manifold  $\Crit (\tilde \cL) \subset \tilde \cC_{1, p}$, there exist a constant $\delta$ and   a small  $W^{1, p}$ 
$\epsilon$-ball neighbourhood $B_\epsilon(x)$ of $x$ in $\tilde \cC_{1, p}$ such that
 \[
 \| \nabla \tilde \cL  (y) \|_{L^2}^2 \geq \delta | \tilde \cL  (y)- \tilde \cL (x)|
 \]
 for any $y\in B_\epsilon(x)$.  Here $\epsilon$ and $\delta$ are independent of $x$
 (assuming that $\mu^{-1}(0)$ is compact).
   \end{proposition}
   
   We remark that though the above inequality is written in a small  $\epsilon$-ball  of a critical point of $\tilde \cL$ on $\tilde \cC_{1, p}$, in fact  the   inequality still holds in a sufficiently small $\epsilon$-ball    of a critical point of $  \cL$ on $  \cC_{1, p}$. This is due to the local diffeomorphism between $\tilde \cC_{1, p}$ and  $\cC_{1, p}$.  That is, the difference
  function 
$
   \cL  (y)-   \cL (x)
  $
   makes sense for $y\in B_\epsilon(x)$ when $\epsilon$ is small.
  
\begin{proof}  
By the gauge invariance, we only need to verify the inequality at critical points of the form
\[
(\exp(\theta \eta ) x_0, \eta )
\]
for  $x_0 \in \mu^{-1}(0)$ and  $\eta  \in \fg$ such that $\exp (2\pi \eta ) \in G_{x_0}$ (a finite group). 
Assume that $\eta=0$, then a  small neighbourhood of $(x_0, 0)$ in $\tilde \cC_{1, p}$ can be
identified with a small ball in 
\[
T_{(x_0, 0)} \cC_{1, p}  = W^{1, p} (S^1, T_{x_0} X\times \fg)
\]
centred at the origin with radius $\epsilon$ for a sufficiently small $\epsilon$. Let $(u, \xi) \in W^{1, p} (S^1, T_{x_0} X\times \fg)$ such that
\[
\|(u, \xi)\|_{W^{1, p}} < \epsilon.
\]
With respect to the canonical metric  defined by $\omega (\cdot, J( \cdot ))$, we have the following $L^2$-orthogonal decomposition
\[
T_{x_0}X  \cong   T_{x_0} \mu^{-1}(0) \oplus \nu_{x_0}
\]
where $\nu_{x_0} =\{J (\tilde \zeta_{x_0}) |   \zeta \in \fg\}$.  
This decomposition provides   a local coordinate  of $X$ at $x_0$, denoted by $(u_0, u_{\mu})$. Under this coordinate, vector fields will be parallelly transported to the origin along the geodesic rays, can be treated as vectors in $T_{x_0}X $. Now we calculate
\[
\|\nabla\tilde \cL (u, \xi) \|_{L^2}^2 = \int_{S^1}  \left( \left\|J\left(\dfrac{du}{d\theta}+ \tilde \xi_u\right)\right\|^2+\|\mu(u)\|^2 \right)d\theta\]
as follows.  Write  $\dfrac{du}{d\theta} = (\dot{u}_0, \dot{u}_\mu)$, we get the following  estimates
\[ \begin{array}{llll} 
&&\disp {\int_{S^1}  \left\|J\left(\dfrac{du}{d\theta}+ \tilde \xi_u\right)\right\|^2 d\theta } & \\[3mm]
& = & \disp {\int_{S^1}  \| (\dot{u}_0 + \tilde \xi_{x_0})  +   \dot{u}_\mu +  (\tilde \xi_{u} -\tilde \xi_{x_0}) \|^2 d\theta} &  
 \\[3mm]
&\geq  &  \disp{\int_{S^1} \left( \| (\dot{u}_0 + \tilde \xi_{x_0}) \|^2 + \|   \dot{u}_\mu \|^2  - \|\tilde \xi_{u} -\tilde \xi_{x_0}\|^2\right) d\theta}
\qquad  & \text{as } \< \dot{u}_0, \tilde \xi_{x_0}\> = \<d \mu_{x_0} (\dot u_0),  \xi \> =0
 \\[3mm]
&\geq &  \disp{\int_{S^1} \left( \|  \dot{u}_0 \|^2   + \|   \dot{u}_\mu \|^2 +  \|\tilde \xi_{x_0}  \|^2  - C \|u\|^2 \|\tilde \xi_{x_0}\|^2 \right)  d\theta} \qquad &   \text{for some constant $C>0$}
\\[3mm] 
&=&  \disp{\int_{S^1}  \left( \|  \dot{u} \|^2  + (1-C\|u\|^2)  \|\tilde \xi_{x_0}\|^2 \right)d\theta }.  
\end{array}
\]
Here $\< \dot{u}_0, \tilde \xi_{x_0}\> = \<d \mu_{x_0} (\dot u_0),  \xi \> =0$ is applied in the calculation. 
Note that 
\[
\int_{S^1} \|\mu (u)\|^2  d\theta \geq \epsilon \|u_\mu\|^2_{L^\infty}, 
\]
for a sufficiently small $\epsilon$. 
Hence, we obtain
\[
\|\nabla\tilde \cL (u, \xi)\|_{L^2}^2 \geq \|\dot u\|^2_{L^2} + \epsilon \|u_\mu\|^2_{L^\infty} + \epsilon \|\xi\|_{L^\infty}^2.
\]
On the other hand, let $\tilde u(\theta, t) = tu(\theta)$ for $t\in [0, 1]$, then
\[
\tilde \cL(u, \xi) - \tilde \cL(x_0, 0) = - \int_{S^1} \tilde u^* \omega + \int_{S^1} \< \mu (u), \xi\>d\theta.
\]
If $\int_{S^1} u_0 d\theta =0 \in T_{x_0} \mu^{-1}(0)$, then
\[
| \int_{S^1} \< \mu (u), \xi\>d \theta | \leq C_1 (\| u_\mu\|^2_{L^\infty} + \|\xi\|^2_{L^\infty}\|)
\]
for some constant $c_1>0$. By a direct calculation,  we know that
\[
| \int_{S^1} \tilde u^* \omega |   \leq C_2 \|u\|_{L^\infty} \int_{S^1} |\dot u| d\theta \leq 
C_3 \|u\|_{W^{1, 2}}, 
\]
for some constants $C_2$ and $C_3$.  
Therefore, for a properly chosen $\delta>0$  and sufficiently small $\epsilon$, we have
\[
\|\nabla\tilde \cL (u, \xi)\|_{L^2}^2 \geq \delta |\tilde \cL(u, \xi) - \tilde \cL(x_0, 0)|.
\]
If   $\int_{S^1} u_0 d\theta \neq 0$, we can replace $x_0$ and $x_0' = x_0 + \int_{S^1} u_0 d\theta $. Then we have
\[
 \tilde \cL(x_0, 0) =  \tilde \cL(x_0', 0).
 \]
 The above calculation applied to $(x_0' , 0) $ implies 
 \[
\left\|\nabla\tilde \cL (u, \xi)\right\|_{L^2}^2 \geq \delta |\tilde \cL(u, \xi) - \tilde \cL(x'_0, 0)|.
\]
So the inequality has been proved for any critical point  which is gauge equivalent to $(x_0, 0)$. 
The above argument can be  adapted for  a   critical point    gauge equivalent to $(\exp(\theta \eta_0) x_0, \eta )$ by lifting technique:
since $Hol(\eta)$ is of finite order, say $k$, then we consider a 
$k$ covering $S^1\to S^1$ and lift $(\exp(\theta \eta_0) x_0, \eta )$ to the first $S^1$. Let $\tilde\eta$ be the lifting of $\eta$ and then $Hol(\tilde\eta)$ is trivial.  Then we can use the above argument to get the estimate.
\end{proof}

\begin{remark}
 In the proof of Proposition \ref{cru:inequ2}, we  use directly  the explicit expression of $\tilde \cL$ and $\nabla\tilde\cL$ in the case that $0$ is  a regular value of $\mu$. If $0$ is irregular, then the argument  in the proof does not apply. We can  apply the Morse-Bott property of the  functional $\tilde \cL$ to establish   a weaker   inequality. In this paper,
 we only deal with the case when $0$ is  a regular value of $\mu$,  so we prefer to use the stronger inequality given by Proposition \ref{cru:inequ2}. 
  \end{remark}

At the end of this subsection, we discuss the energy of gradient flow line.
Let $\gamma = (\tilde x,\tilde \eta) :  [a, b] \to \cC_{1, p}$ be a path connecting
$(x_1,\eta_1)$ and $(x_2,\eta_2)$. Let $(\tilde x_1,\tilde \eta_1)$
be a path $\gamma_1$ connecting the based point to $(x_1,\eta_1)$; then we set $(\tilde x_2,\tilde \eta_2)$ be the path 
$\gamma_2=\gamma_1\sharp\gamma$.  As in Remark \ref{cover}, we treat $(\tilde x_1, \tilde \eta_1)$ and $(\tilde x_2,\tilde \eta_2)$
as elements in $\tilde\cC$.

\begin{lemma} \label{energy}
Suppose that $\gamma= \gamma(t): [a, b] \to \cC_{1, p}$  
is a gradient flowline of $\cL$. 
Then the following  quatiyare equal:
\begin{enumerate}
\item $\tilde{\mc L}(\tilde x_1,\tilde\eta_1)-\tilde{\mc L}(\tilde x_2,\tilde\eta_2)$
\item  $ -  \disp{\int_{S^1\times [a, b] } \tilde x^*\omega + \int_{S^1} (\<\mu(x_1 ), \eta_1 \> -  
\<\mu(x_2), \eta_2 \>d\theta} $
\item the Yang-Mills-Higgs energy $E(\tilde x,\tilde \eta)$;
\item $ \disp{\int_a^b} \left\|\dfrac{\p \gamma (t)}{\p  t} \right\|^2_{L^2} dt.$
\end{enumerate}
\end{lemma}
\begin{proof}
Recall that a path $\gamma = (x, \eta)$ is a gradient flow line of $\nabla \cL$  if it satisfies the equations
\[
\dfrac{\p}{\p t} \left(x(t), \eta(t) \right)  = -  \left( J( \dfrac{\p x}{\p \theta} + \tilde \eta_x ),  \mu(x) \right).
\]
Since $(\tilde x,\tilde \eta)$ solves the symplectic vortex equation, 
$$
E(\tilde x,\tilde \eta)=
-\int_{S^1\times [a,b]} \tilde x^\ast\omega+d\langle\mu(\tilde x),\tilde \eta\rangle.
$$
This implies that (1)=(2)=(3). Now we show that (1)=(4).
$$
\tilde{\mc L}(\tilde x_2,\tilde\eta_2)-\tilde{\mc L}(\tilde x_1,\tilde\eta_1)=\int_{a}^b\frac{d}{dt}\mc L(\tilde x(t),\tilde \eta(t))dt
=\int_a^b\<\nabla\mc L, \frac{d\gamma}{dt}\>dt= \int_a^b \left\|\dfrac{\p \gamma (t)}{\p  t} \right\|^2_{L^2} dt.$$
\end{proof}

  \subsection{Asymptotic behaviour of finite energy  symplectic vortices on a cylinder} \ 
  
  In   this subsection, we   establish  the existence of a limit point for any gradient flow line 
  \[
  \gamma: [0, \infty) \to \cC_{1, p}
  \]
 with finite  energy $E(\gamma)$.
   Then by Lemma \ref{energy}, we have
   \ba\label{equ:1}
   \int_0^\infty \left\|\dfrac{\p \gamma (t)}{\p  t} \right\|^2_{L^2} dt = 
    \int_0^\infty \|
   \nabla \cL (\gamma (t))   \|^2_{L^2} dt = E(\gamma) < \infty.
\na

     \begin{theorem}\label{decay:exp} Let  $\gamma: [0, \infty) \to \cC_{1, p}
 $ be a gradient flow line of $\cL$   with finite  energy. Then there exists a unique critical point 
 $(x_\infty, \eta_\infty) \in \Crit (\cL)$ and  constants $\delta, C>0$ such that the $L^2$-distance 
 \[
 \dist_{L^2} \big( \gamma (T), (x_\infty, \eta_\infty)\big) \leq C e^{-\delta T}
 \]
  for  any sufficient large $T$. 
     \end{theorem}
    \begin{proof}   {\bf Step 1.}  For any sequence $\{\gamma (t_i) | \lim_{i\to \infty} t_i =\infty\}$, we show that  there exists 
    a convergent  subsequence, still denoted by $\{\gamma (t_i)\}$,  such that up to gauge transformations in
    $\cG_{2, p}$,  the sequence $\{\gamma (t_i)\}$ converges to   a critical point $y_\infty$ of $\cL$ in the  $C^\infty$-topology.   
    
    Let $(u_i, A_i) = \gamma(t)$ be the symplectic vortex on $S^1\times [-1, 1]$ in temporal gauge, obtained from 
    $\gamma:  [ t_i -1 , t_i +1 ] \to \cC_{1, p}$. Then we have 
   \[
   \lim_{i\to \infty} E(u_i, A_i) =0,
   \]
where  the energy $E(u_i, A_i)$ agrees with the Yang-Mills-Higgs energy
   \[
   E(u_i, A_i) = \int_{S^1\times [-1, 1]} \dfrac 12 \left( |d_{A_i} u_i|^2 + |F_{A_i}|^2 + |\mu(u_i)|^2\right) d\theta d t.
   \]
   Applying the standard regularity result and Uhlenbeck compactness, see Theorem 3.2 in \cite{CGRS}, there exists a sequence of  $W^{2, p}$-gauge transformations $g_i$ of $P_{S^1}\times [-1, 1]$ such that the sequence 
   \[
   g_i \cdot (u_i, A_i)
   \]
    has a $C^\infty$-convergent subsequence. Let $(u_\infty, A_\infty)$ be the limit, then $(u_\infty, A_\infty)$
    satisfies the following equations
    \ba\label{equ:end}
    F_{A_\infty } =0 , \qquad 
  d_{A_\infty} u_\infty =0, \qquad 
   \mu(u_\infty) =0.
      \na
   We can find a smooth gauge transformation $h$ of $P_{S^1}\times [-1, 1]$ such that 
 $ h \cdot (u_\infty, A_\infty)$ is in temporal gauge. So   we can write 
 \[
 h \cdot (u_\infty, A_\infty) = (x_\infty(\theta, t), \eta_\infty (\theta, t) d\theta)
 \]
as   a path in $\cC_{1, p}$. Then  the equations (\ref{equ:end}) become 
 \[ \left\{ \begin{array}{lll}
 \dfrac{\p \eta_\infty(\theta, t)}{\p  t} =0,    \dfrac{\p x_\infty(\theta, t)}{\p  t} =0
  \\[3mm]
    \dfrac{\p x_\infty}{\p \theta}
 + \tilde \eta_{\infty} (x_\infty) =0,  \mu (x_\infty) =0. &&
 \end{array}\right.
 \]
 These equations  imply that $(x_\infty, \eta_\infty) =  h \cdot (u_\infty, A_\infty) \in \Crit (\cL)$ and 
 \[
 \lim_{i\to \infty} (hg_i) \cdot (u_i, A_i) = (x_\infty, \eta_\infty) 
 \]
 in the  $C^\infty$-topology.  Hence, up to gauge transformations in
    $\cG_{2, p}$,  the subsequence $\{\gamma (t_i)\}$ converges to   a critical point $(x_\infty, \eta_\infty) $ of $\cL$ in the  $C^\infty$-topology.  We denote it by $y_\infty$.
    
  {\bf Step 2. } Set $\gamma^i=g_i\gamma$. We claim that there exists $i$ such that $\gamma^i(t)\in B_\epsilon(y_\infty)$ for $t\geq t_i$. Here $\epsilon$ is the constant given in Proposition \ref{cru:inequ2}.
  
If not, for each $i$ there exists $s_i>t_i$ such that the path $\gamma^i(t), t_i\leq 
t\leq s_i$ locates in $B_\epsilon(y_\infty)$ and $\gamma^i(s_i)
\in\partial B_\epsilon(y_\infty)$. 
Then
\begin{eqnarray*}
\mathrm{dist}_{L^2}(\gamma^i(t_i),\gamma^i(s_i))
&\leq &\int_{t_i}^{s_i}\|\frac{\partial \gamma^i(t)}{\partial t}\|_{L^2}dt
=\int_{t_i}^{s_i}\|\nabla(\mc L(\gamma^i(t)-\mc L(y_\infty))\|_{L^2}dt\\
&\leq& 
-2\delta^{-1}\int_{t_i}^{s_i}\frac{d}{dt}((\mc L(\gamma^i(t)-\mc L(y_\infty))^{1/2})\\
&\leq &2\delta^{-1}((\mc L(\gamma^i(t_i)-\mc L(y_\infty))^{1/2}-(\mc L(\gamma^i(s_i)-\mc L(y_\infty))^{1/2}).
\end{eqnarray*}
As $i\to \infty$, this goes to 0.  On the other hand, by Step 1, there exists $h_i$
such that $h_i\gamma^i(s_i)$ uniform converges to some critical point $y'_\infty$.  Since $\gamma^i(s_i)\in B_\epsilon(y_\infty)$
and $h_i\gamma^i(s_i)$ uniformly converges to $y_\infty'$, $h_i$
is uniformly bounded at least in $C^{1,\alpha}$ for some $\alpha>0$. This means that there exists a subsequence of $h_i$ that converges. We may relabel the sequence and assume that $h_i$ converges to $h$. We conclude that $\gamma^i(s_i)$ converges to $h^{-1}y_\infty'$. Therefore,
$$
\dist_{L^2}(\gamma^i(t_i),\gamma^i(s_i))\to 
\dist_{L^2}(y_\infty,h^{-1}y_\infty') =0.
$$
This implies that $y_\infty=h^{-1}y_\infty$. However, $\gamma^i(s_i)$
is on the boundary of the ball $B_\epsilon(y_\infty)$, this is impossible. The contradiction implies Step 2.

  {\bf Step 3:} From Step 2, suppose that $\gamma^i(t)$ locates in 
  $B_\epsilon(y_\infty)$ when $t$ large. Reset $y_\infty$ to be $g^{-1}_iy_\infty$. Then we may assume that  $\gamma(t)$ locates in $B_\epsilon(y_\infty)$ when $t$ large. Now we show that 
  $$
  \mathrm{dist}_{L^2}(\gamma(t),y_\infty)\leq Ce^{-\delta t}
  $$
  for $t$ large.
  
We can assume that for $ t>T_0$,  $\gamma (t)\in B_\epsilon ( y_\infty)$ so that the crucial inequality in Proposition \ref{cru:inequ2}  can be applied to
 get 
 \[
 \dfrac{d \big(    \cL ( \gamma (t))-    \cL ( y_\infty)\big)^{1/2}}{dt}  = -\|\nabla  ( \cL ( \gamma (t))\|_{L^2} 
  \leq -\delta \big(\cL ( \gamma (t))-     \cL (  y_\infty)\big)^{1/2}.
  \]
Hence, for any $t>T_0$, we have 
 \[
   \cL  ( \gamma (t)) -  \cL  ( y_\infty)    \leq 
  e^{-\delta (t-T_0)}  \big( \cL  ( \gamma (T_0)) -  \cL  ( y_\infty)  \big).
  \]
 That is,  for $t>T_0$
  \[
  \dist_{L^2} \big( \gamma (t),  y_\infty\big) \leq 2c e^{-\delta(t-T_0)}  \big(  \cL  ( \gamma (T_0)) -\cL  (  y_\infty)  \big)^{1/2}.
  \]
  Take $C= 2c e^{T_0}   \big(  \cL  ( \gamma (T_0)) -\cL  ( y_\infty)  \big)^{1/2}$, we get the exponential decay estimate for $ \dist_{L^2} \big( \gamma (t),  y_\infty \big)$.
       \end{proof}
    
     \begin{remark} By a similar calculation  as   the proof, one can establish the following exponential decay for a  finite energy gradient flow line $\gamma:  [0, \infty) \to \cC_{1, p}$. That is,  there exist  constants $\delta, C>0$ such that 
     \[
     \disp{\int_{T}^{\infty}  }  \|\nabla \cL (   \gamma (t))\|^2_{L^2} dt  
  \leq Ce^{-\delta T}
      \]
      for a sufficiently large $T$. Moreover, let $y_\infty$ be the limit of $\gamma (t)$ at infinity, by a gauge transformation, we may assume that $y_\infty\in \Cr$, then for any
      $k\in \N$, there exist $C, \delta >0$ such that
      \ba\label{ell:estimate}
      |\nabla^k \gamma (t) | \leq C e^{-\delta t}
      \na
      for $t$ sufficiently large.  To get the above point-wise estimate, we apply the elliptic regularity to
  the symplectic vortex    $\gamma |_{ [T-2, T+2]\times S^1} $ for a sufficiently large $T$ to get
  a $C^k$-estimate
  \[
  \| g\cdot \gamma \|_{C^k} \leq C.
  \]
  for  some constant $C>0$ and any $k\in \N$. Write $\gamma = (\alpha, u)$, then the curvature
  $F_\alpha$ and $\mu(u)$ are  gauge invariant and hence bounded.  Then (\ref{ell:estimate}) follows from 
  applying the standard elliptic estimates to the gradient flow equations.  We also remark that the decay rate $\delta$ can be chosen such that $\delta$ is smaller than  the minimum  absolute value of
 non-zero eigenvalues of the Hessian operator of $\cL$ at $y_\infty$. 
      \end{remark}

 \section{$L^2$-moduli space of   symplectic vortices on a cylindrical  Riemann surface}

  In this section, we consider the symplectic vortices of finite energy  on a Riemann surface $\Sigma$  with cylindrical end. For simplicity, $\Sigma$ is assumed to have just one end, isometrically diffeomorphic to a half cylinder $S^1 \times [0, \infty)$. Let $K$ be a compact set of $\Sigma$ such that $\Sigma \backslash K$ is  isometrically diffeomorphic to  $S^1\times (1, \infty)$ with the flat metric $d\theta^2 + dt^2$.  Let $P$ be a principal $G$-bundle over $\Sigma$ and $\cN_{\Sigma}(X, P)$ be the moduli space of symplectic vortices with  finite energy  associated to
  $P$ and a closed Hamiltonian manifold $(X, \omega)$.  It is the space of gauge equivalence classes of 
  \[
  (A, u) \in \cA(P) \times C^\infty_G (P, X)
  \]
  satisfying the symplectic vortex equations  (\ref{sym:vortex}) and with the property 
that the Yang-Mills-Higgs energy  (Cf. (\ref{YMH:energy})) is finite. The main result of this section is to prove  the  continuity for the asymptotic limit map established in Section \ref{3} 
  \[
  \p_\infty:  \cN_{\Sigma}(X, P)\longrightarrow \Cr,
  \]
  where $\Cr$ is the critical point set modulo the  gauge transformations. We remark that
  $\Cr$ is  diffeomorphic to the inertia orbifold $I\cX_0$ associated to the reduced symplectic orbifold $\cX_0 = [\mu^{-1}(0)/G]$, as we assume that $0$ is a regular value of the moment map $\mu$. 
   
To study the moduli space $\cN_{\Sigma}(X, P)$, we consider the  $W^{1, p}$-space
\[
\widetilde \cB_{W^{1, p}_\loc (\Sigma)} = \cA_{W^{1, p}_\loc(\Sigma)} \times W^{1, p}_{\loc, G}(P, X)
\]
for $p\geq 2$. Then by the elliptic regularity,  $\cN_{\Sigma}(X, P)$ 
is the space of solutions to the symplectic vortex equations  (\ref{sym:vortex}) for
$(A, u) \in\widetilde  \cB_{W^{1, p}_\loc(\Sigma)}$ such that
\[
E(A, u) = \int_{\Sigma} \dfrac 12 (|d_Au|^2 + |F_A|^2 + |\mu\circ u|^2) \nu_\Sigma <  \infty
\] 
modulo the  the action of the  group $\cG_{W_{loc}^{2, p}(\Sigma)}$ of all  $W^{2, p}_\loc$ gauge transformations.

   \begin{proposition} 
   \label{asymp:map} Let $\Sigma$ be a Riemann surface $\Sigma$  with  one cylindrical end,   $P$ be a principal $G$-bundle over $\Sigma$ and $\cN_{\Sigma}(X, P)$ be the moduli space of symplectic vortices with  finite energy  associated to
  $P$ and a closed Hamiltonian manifold $(X, \omega)$. Then the asymptotic limit of  symplectic vortices in 
   $\cN_{\Sigma}(X, P)$ define a continuous map
   \[
    \p_\infty:   \cN_{\Sigma}(X, P) \longrightarrow\Cr \cong  I\cX_0.
    \]
   \end{proposition}
\begin{proof} Let $[(u, A)] \in   \cN_{\Sigma}(X, P)$ and $\p_\infty([(u, A)] ) = y_\infty$ correspond to
$[x_0]\in (\mu^{-1}(0))^g /C(g)$. Fix  an open neighbourhood  $ V$  of $[x_0] $ in  $(\mu^{-1}(0))^g /C(g)$, a twisted sector in $I\cX_0$. We need to find an open neighbourhood  $ U \subset  \cN_{\Sigma}(X, P)$  of $[(u, A)]$ such that 
\[
\p_\infty ( U) \subset  V.
\]
Let $\tilde \cV$ be a $\cG_{2, p} (S^1)$-invariant  open neighbourhood of $(\exp(2\pi\theta\eta_0)x_0, \eta_0)$ in
$\cC_{1, p}$ such that $\tilde \cV \cap \Crit (\cL)$ is mapped to  a subset of $V$ under the identification 
\[
\Crit(\cL)/\cG_{2, p}(S^1)  \cong I\cX_0.
\]
Denote by $\tilde \cU$ the solutions $(u, A) \in  \widetilde \cB_{W^{1, p}_\loc (\Sigma)} $ to the symplectic vortex equations with finite energy such that for sufficiently large $T$, the restriction of
$(u, A)$ to $S^1\times [T, \infty)$ is gauge equivalent to an element in $\tilde \cV$. Then 
$U= \tilde \cU/ \cG_{W^{2, p}(\Sigma)}$ is an open neighbourhood of $[(u, A)] $ in   $\cN_{\Sigma}(X, P)$ and
$\p_\infty ( U) \subset  V$. 
\end{proof}

 \subsection{Fredholm theory for $L^2$-moduli space of   symplectic vortices}

To understand the moduli space $ \cN_{\Sigma}(X, P)$, we need to introduce the weighted Sobolev space for the fiber of the e asymptotic limit map 
\[
  \p_\infty:   \cN_{\Sigma}(X, P) \longrightarrow \Cr \cong  I\cX_0.
  \]
   Any 
symplectic vortex $[(u, A)]\in  \cN_{\Sigma}(X, P)$ decays exponentially to its asymptotic limit
$  \p_\infty([(u, A)])$ at a rate  $\delta>0$ for some $\delta$ such that $\delta$  is smaller than 
 the  minimum  absolute value of  non-zero eigenvalues of the Hessian operator of $\cL$ at $ \p_\infty([(u, A)])$. 
Note that  $\Cr$ is compact, so we can choose a constant  $ \delta$ such that  
 $[(u, A)]\in  \p_\infty^{-1}(y_\infty)$  decays exponentially to  $y_\infty$ at the rate $\delta$ for any $ y_\infty
 \in   \Cr$.  We fix such a $\delta$ throughout this section.

  Fix a smooth function $\beta:  \Sigma \to [0, \infty)$ such that the follow conditions hold: 
  \begin{enumerate}
\item On $ S^1 \times [1, \infty)$, $\beta$ is the coordinate function on  the cylinder. 
\item $\beta =0$ on $\Sigma \setminus\{ S^1\times [0, \infty)\}$.
\item  $\beta|_{ S^1 \times [0, 1]}$ is an increasing function. 
\end{enumerate}
The weighted  $W^{k,p}$-norm on  a compact support section $\xi$  of an Euclidean vector bundle
$V$ over $\Sigma$ with a covariant derivative $\nabla$ is defined by 
\[
\|\xi\|_{W^{k, p}_\delta} =\left( \int_\Sigma e^{\delta\beta} \big(
| \xi |^p + |\nabla (\xi)|^p  +\cdots + |\nabla^p (\xi)|^p\big) d\nu_\Sigma\right)^{1/p}.
\]
We denote  $W^{k, p}_\delta (\Sigma, E)$ the completion of all  compact support sections  of $E$ with respect to the weighted  $W^{k,p}$-norm,    which is also called the Banach space of $W^{k, p}_\delta$-sections of $E$. When $k=0$,  we simply  denote by $L^p_\delta (\Sigma, E)$ the $W^{0, p}$-sections of $E$. 

Let  $(A_\infty, u_\infty) \in \Crit(\mc L)$. Let
by pulling back, we get a  $(A_0, u_0)$ on cylinder end $S^1\times
[1,\infty)$, namely it is constant in $t\in [1, \infty)$ and agrees with
$(A_\infty, u_\infty)$. 

Define
$ \widetilde \cB_\delta (A_\infty, u_\infty) $ be the subspace of $ \widetilde \cB_{W^{1,p}_{loc}}$ consisting of elements $(A, u)$ with the following property:
\begin{enumerate}
\item $A-A_0 \in W^{1, p}_\delta (S^1\times [1,\infty),   \Lambda^1 \otimes  P^{ad})$,
\item there exists a  sufficiently large $T$  depending on $(A, u)$  such that 
$
u|_{S^1\times [T, \infty)} =\exp_{u_\infty} (v)
$
for  $v \in W^{1, p}_\delta (S^1\times [T, \infty), u_\infty^* TX)$. 
 \end{enumerate}
 Then $\widetilde \cB_\delta (A_\infty, u_\infty)$ is a Banach manifold whose tangent space at $(A, u)$ is
 given by
 \[
 T_{(A, u)} \widetilde \cB_\delta (A_\infty, u_\infty) =  W^{1, p}_\delta (\Sigma,   \Lambda^1 \otimes  P^{ad} \oplus u^*TX).
 \]
 Define
 $$
 \widetilde{\mc B}_\delta=\bigcup_{(A_\infty,u_\infty)} \widetilde{\mc B}_\delta(A_\infty,u_\infty)\to \Cr(\mc L).
 $$
This is  a smooth family of Banach manifold.

 The gauge group $\cG_\delta$  in this setting is defined  as follows.
 For each $g_\infty\in \mc G(P_{S^1})$, we construct a $g_0$ on $S^1\times [1,\infty)$ by pulling back $g_\infty$. Then we define
 $\mc G_\delta(g_\infty)$ similarly using $W^{2,p}_\delta$-norm. 
 Set $\cG_\delta$ be the union of $\mc G_\delta(g_\infty)$.  Set
 $$
 \mc B_\delta=\widetilde{\mc B}_\delta/\mc G_\delta.
 $$
 The symplectic vortex equation (\ref{sym:vortex})  defines a smooth  $\cG_\delta $-invariant  section  $S$ of the $\cG_\delta$-equivariant  Banach bundle   $\widetilde\cE_\delta$ whose fiber at $(A, u)$ is given by
 \[
  \big( \widetilde \cE_\delta (A_\infty, u_\infty)\big)_{(A, u)} = L_\delta^p (\Sigma, \Lambda^{0,1} \otimes  u^*T^\vert Y \oplus    P^{ad})
  \]
where $Y = P\times_G X$, where $(A,u)\in \widetilde{\mc B}_\delta
(A_\infty,u_\infty)$. The moduli space 
$$
\mc N_\Sigma(X,P)=S^{-1}(0)/\mc G_\delta.
$$

 We explain $\partial_\infty^{-1}(y_\infty)$. Let $(A_\infty,u_\infty)$
 be a representative of $y_\infty$. 
  Note that by our assumption,  $(A_\infty, u_\infty) $ has only a finite group as its isotropy group, denoted by $G_\infty$. Then we may define $\cG_\delta  (A_\infty, u_\infty)$
is a disjoint union of $\mc G_\delta(g)$ for $g\in G_\infty$. The connected component of the  identity is the Banach Lie group whose Banach Lie algebra is $W^{2, p}_\delta (\Sigma, P^{ad})$. 

The symplectic vortex equation (\ref{sym:vortex})  defines a smooth  $\cG_\delta  (A_\infty, u_\infty)$-invariant  section  $S_{(A_\infty, u_\infty)}$ of the $\cG_\delta  (A_\infty, u_\infty)$-equivariant  Banach bundle   $\widetilde\cE_\delta$ whose fiber at $(A, u)$ is given by
 \[
  \big( \widetilde \cE_\delta (A_\infty, u_\infty)\big)_{(A, u)} = L_\delta^p (\Sigma, \Lambda^{0,1} \otimes  u^*T^\vert Y \oplus    P^{ad})
  \]
where $Y = P\times_G X$.   The deformation complex associated  to a symplectic vortex $(A, u) \in \widetilde \cB_\delta (A_\infty, u_\infty)$
  is given by
  \[
\xymatrix{
W^{2, p}_\delta (\Sigma, P^{ad})  \ar[r]^{L_{A, u}\qquad \qquad  } &  W^{1, p}_\delta (\Sigma,   \Lambda^1 \otimes  P^{ad} \oplus u^*TX)
 \ar[r]^{\cD_{A, u} }& 
   L_\delta^p (\Sigma, \Lambda^{0,1} \otimes  u^*T^\vert Y \oplus    P^{ad})
   }, 
\]
which is elliptic in the sense that the cohomology groups are finite dimensional.  The proof of this statement is quite standard  nowadays so we omit it here.  See the books \cite{Don1},   \cite{MMR}  and \cite{Tau}.  This ensures that the quotient of
\[
\big(  \widetilde \cB_\delta (A_\infty, u_\infty), \widetilde \cE_\delta (A_\infty, u_\infty), S_{(A_\infty, u_\infty)}\big)
\]
the gauge group  $\cG_\delta  (A_\infty, u_\infty)$  is a Fredholm system. By  the exponential decay result  for  symplectic vortices  in  $\cN_{\Sigma}(X, P)$, we know that 
it  is indeed  a Fredholm system for the fiber of the asymptotic limit map at $(A_\infty,u_\infty)$. 
  The index of this Fredholm system  can be computed by the Atiyah-Patodi-Singer index formula  for elliptic differential operators on manifolds with cylindrical end. The  theorem below  relates this index with the expected dimension for orbifold symplectic vortices.
  
Note that the formal dimension of the moduli space $\cN_{\Sigma}(X, P)$ at a point $[(A, u)]$ is given by 
the  the index of the operator  $\cD_{A, u} \oplus L^*_{A, u}$ associated to  the  $(-\delta)$-weighted  deformation complex.   For the purpose of the calculation,  we can  replace the operator for 
 $(A, u)$ on $\Sigma$ to  by a  suitable operator on an associated orbifold Riemann surface. We construct this replacement as follows. Assume that 
 \[
 \p_\infty (A, u) =  (A_\infty, u_\infty) = (\xi d\theta,\exp(\theta \xi)\cdot  x_\infty)
 \]
where $g =\exp(2\pi \xi)$ has order $m$ and   $x_\infty \in \big(\mu^{-1}(0)\big)^g$. Note that 
$[(A_\infty, u_\infty)]$ determines an element in $ \big(\mu^{-1}(0)\big)^g/C(g) \subset \Cr$, a twisted sector of
the reduced orbifold $\cX_0 = \mu^{-1}(0)/G$.  We can identify the cylinder $ S^1 \times [0,  \infty)$ with a
unit disc $\D^* = \D-\{0\}$ in $\C$ using the coordinate change $(i\theta, t) \mapsto e^{-(t+ i\theta)}$. Then the cylindrical surface 
$\Sigma$ become a punctured Riemann surface. Denote this punctured  Riemann surface  by $\Sigma^*$.  Let $P^*$ be the corresponding principal $G$-bundle over $\Sigma^*$. 

 As the connection $A_\infty$ has a non-trivial holonomy,  $A$ does not extend to a connection $P^*$. 
Consider the degree $m$ covering map 
$
\phi:   \D-\{0\} \longrightarrow  \D-\{0\}
$
defined by $re^{i\theta} \mapsto (re^{i\theta})^m$. Then 
$\phi^* (A_\infty, u_\infty)$, as a rotation invariant symplectic vortex on $\D^*$
is  gauge equivalent to   $(0, x_\infty)$ which is   
a $\Z_m$-invariant symplectic vortex on $ \D$, or extends to a constant symplectic vortex on the 
orbifold $[\D/\Z_m]$ associated to a trivialized principal $G$-bundle. 

Let $\Sigma_{orbi} $ be the orbifold Riemann surface whose underlying topological
space is the closure of $\Sigma^*$.  The punctured point $p$  is treated as an singular point  
locally modelled on $(\D, \Z_m)$ with the action of $\Z_m$ on $\D$ is generated by the multiplication 
$(e^{2\pi i /m},  z)\mapsto   e^{2\pi i /m}  z$. 
Then   the above discussion implies that   $(A, u)$ can be replaced by    a pair $(\tilde A, \tilde u)$ on  the orbifold Riemann surface associated to $(P_{orbi}, X, \omega)$ such that the  chosen trivialization at the orbifold
point $p$ is specified by a based point $\tilde p $ in the fiber of $P_{orbi}$, and 
\[
\tilde u (\tilde p) = x_\infty\in ( \mu^{-1}(0))^g,
\] 
where $g=\exp(2\pi \xi)$ and $x_\infty$  are determined by the asymptotic value 
as above. 
Hence,  $\tilde u$ gives rise to a degree 2 equivariant  homology class  in $H_2(X_G, \Z)$. For simplicity,
we still denote this class by 
$[ u_G]$ which is called the {\bf homology class of $(A, u)$}. Then by a direct calculation,  the energy of $(A, u)$ is 
\[
E(A, u) = \< [\omega -\mu], [ u_G]\>.
\]

Fix an equivariant homology class $B\in H_2(X_G, \Z)$ such that $\< [\omega -\mu], B\> > 0$. Let $\cN_\Sigma (X, P, B)$ be the moduli space of symplectic vortices on $\Sigma$ associated to $(P, X)$ with the homology class $B$. Then
$\cN_\Sigma (X, P, B) \subset \cN_\Sigma (X, P)$ and  the asymptotic limit map in Proposition \ref{asymp:map}
defines  a continuous map on  $\cN_\Sigma (X, P, B)$
\[
\p_\infty: \cN_\Sigma (X, P, B) \longrightarrow \Cr. 
\]

To calculate the expected dimension of components in  $\cN_\Sigma (X, P, B)$ over $(\mu^{-1}(0))^g/C(g)$, we introduce a degree shift as in \cite{CR}. 
 We first  define  the   degree shift  of  an element  $g$ in $G$ of order $m$ acting on $\C^n$.  Let the complex
  eigenvalues of $g$ on $\C^n$ be
  \[
  e^{2\pi i m_1/m},   e^{2\pi i m_2/m}, \cdots,  e^{2\pi i m_n/m}
  \]
 for an $n$-tuple of integers $(m_1, m_2, \cdots, m_n)$ with $0\leq m_j < m$ for $j=1, 2, \cdots, n$. Then 
  the   degree shift  of  an element  $g$ on $\C^n$, denoted by $\iota (g, \C^n)$, is given by
  \[
  \iota (g, \C^n) = \sum_{j=1}^n \dfrac{m_j}{m}.
  \]
  From the definition, we have
  \[
   \iota (g, \C^n)  +   \iota (g^{-1}, \C^n) = n.
   \]
  For the orbifold $\cX_0 = [\mu^{-1}(0)/G]$, if $g\in G$ has a non-empty fixed point set  $(\mu^{-1}(0))^g$, then 
  the Chen-Ruan degree shift of $g$ on $\cX_0$, denoted by $\iota_{CR }(g,  \cX_0)$ at $x  \in (\mu^{-1}(0))^g$,
   is defined to be
   \[
   \iota_{CR}(g,  \cX_0) =  \iota(g, T_{[x]} (\cX_0) ).
   \]
   For a twisted sector $\cX^{(g)}_0$ of $\cX_0$, the corresponding degree shift as in \cite{CR} is
   defined to 
   \[
   \iota_{CR}( \cX_0^{(g)},  \cX_0) =  \iota_{CR}(g,  \cX_0) 
   \]
   for any $g$ such that  $\cX_0^{(g)} $ is diffeomorphic to the orbifold defined by the action of 
   $C(g)$ on the fixed point manifold $ (\mu^{-1}(0) )^{g}$.

  \begin{theorem}\label{APS:index}  Let $ \cN_{\Sigma} (X, P, B;  (g ) )$ be the subset of
   $\cN_{\Sigma} (X, P, B)$ consisting of 
symplectic vortices  $[(A, u)] $ such that 
  \[
  \p_\infty  (A, u) \in    \cX_0^{(g)}    \subset  \Cr 
     \]
Then   $\cN_{\Sigma} (X, P, B; (g ) )$  admits a Fredhom system with its virtual dimension given by 
\[
2\<c_1^G(TX), B\> + 2(n-\dim G) (1-g_{\Sigma}) - 2   \iota_{CR}( \cX_0^{(g)},  \cX_0)
\]
where $g_{\Sigma}$ is the genus of the Riemann surface $\Sigma$. 
 \end{theorem}
  
  \begin{proof}  With the Fredholm set-up and the orbifold model discussed above, we only need to calculate the index of
  the linearisation operator for the symplectic vortex $(\tilde A, \tilde u)$ on the orbifold Riemann surface
  $\Sigma_{orib} = (|\Sigma_{orbi}|,   (p, m))$, modulo based gauge transformations.  Note that the remaining gauge transformations consist of constant ones taking values in $C(g)$, the centraliser of $g$ in $G$, as $\tilde u (\tilde p ) 
 \in \big(  \mu^{-1}(0)\big)^g$. 
  
  The underlying Fredholm operator 
  is the a compact perturbation of the direct sum of the operator $(-d^*_A,  *d_A)$ 
  \[
  \Omega^1_\delta (\Sigma, P^{ad}) \to  \Omega^0_\delta (\Sigma, P^{ad})
 \oplus   \Omega^0_\delta (\Sigma, P^{ad})
 \]
 in the original cylindrical model, 
 with its  index given by  $-\dim G (1-2g_\Sigma)$, and the Cauchy-Riemann operator $\bar{\p}_{\tilde A, \tilde u}$ on 
 the orbifold $\Sigma_{orbi}$ with values in  the complex vector bundle  $\tilde u^* T^{\vert} Y$.
 Hence, the virtual dimension is given by
 \ba\label{index:sum}
  \ind \bar{\p}_{\tilde A, \tilde u} - \dim G (1-2g_\Sigma) -\dim C(g).
  \na
  
  By the orbifold index theorem,
 we have
 \ba\label{index:CR}
 \ind \bar{\p}_{\tilde A, \tilde u} =  2\<c_1 (u^* T^{vert} Y), [|\Sigma_{orbi}|]\>  + 2n  (1-g_\Sigma) - 
 2 \iota_{CR}(g,  T_{\tilde u (\tilde p)} X).
 \na
 By the definition of $c_1^G(u^* T^{vert} Y)$ and $[u_G]$, we have 
 \[
 \<c_1 (u^* T^{\vert} Y), [|\Sigma_{orbi}|]\> = \<c_1^G(TX), B\>.
 \]

To calculate the degree shift for the $g$-action on $T_{\tilde u (\tilde p} X$, we apply the following 
decomposition
\[
T_{\tilde u (\tilde p)} X \cong  \fg \oplus\fg^* \oplus T_{[x_\infty]} \cX_0.
\]
  Here the actions of $g$ on  $\fg$ and $\fg^*$ are adjoint to each other and the zero eigenspace of
  the $g$-action on $\fg$ is the Lie algebra of $C(g)$. By the definition of degree shift, this implies that
  \ba\label{sum}\begin{array}{lll}
&&  2 \iota_{CR}(g,  T_{\tilde u (\tilde  p) } X) \\[2mm]
&=&  2\iota_{CR}(g,  \cX_0) +2 \dim_\C G/C(g)\\[2mm]
&=&   2\iota_{CR}(g,  \cX_0) +  \dim_\R G/C(g).
\end{array}
\na
Put these formula (\ref{index:sum}),  (\ref{index:CR}) and (\ref{sum}) together, we get the virtual dimension
as claimed in the theorem. 
\end{proof}


\subsection{$L^2$-moduli space of   symplectic vortices on punctured Riemann surface}\

Let $C=(\Sigma, p_1,\ldots, p_k)$ be a Riemann surface with $k$ marked  points.    We assume that $C$ is stable, i.e, $2-2g(\Sigma) -k <0$ where $g_{\Sigma}$ is the genus of the Riemann surface $\Sigma$. It is well known that there is a canonical hyperbolic metric on   the punctured Riemann surface 
$\Sigma\setminus\{p_1,\ldots,p_k\}$.  This hyperbolic metric  provides  a disjoint  union of  horodiscs $D(p_i)$ rounded at 
each punctured point $p_i$. We may deform the metric on the disc such that the metric becomes
a cylinder end metric.  For simplicity, we use the same notation $\Sigma$ for this Riemann surface with   $k$ cylindrical  ends.  Denote the metric and the corresponding  volume form   by $\rho_{\Sigma}$ and $\nu_{\Sigma}$ respectively.

 Let $P$ be a  principal $G$-bundle over $\Sigma$.  Let $\cN_{\Sigma}(X, P)$  be the moduli space of symplectic vortices with finite energy on $\Sigma$ associated to $P$ and a  $2n$-dimensional Hamiltonian $G$-space $(X, \omega)$.  
Then $\cN_{\Sigma}(X, P)$ 
is the space of  gauge equivalence classes of solutions to the symplectic vortex equations  (\ref{sym:vortex}) for
\[
(A, u) \in\widetilde  \cB_{W^{1, p}_\loc(\Sigma)} =  \cA_{W^{1, p}_\loc(\Sigma)} \times W^{1, p}_{\loc, G}(P, X)
\]
 such that
\[
E(A, u) = \int_{\Sigma} \dfrac 12 (|d_Au|^2 + |F_A|^2 + |\mu\circ u|^2) \nu_{\Sigma} <  \infty.
\]
Then  the asymptotic limit map  
\[
\p_{\infty}:  \cN_{\Sigma }(X, P) \longrightarrow \big(\Cr\big)^k 
\]
is continuous.  Let $\delta$ be a positive real number which is smaller than the minimum absolute 
value of eigenvalues of the Hessian operators of $\cL$ along the compact critical manifold $\Cr$, then 
 $[u, A]\in  \cN_{\Sigma}(X, P)$ decays exponentially to its asymptotic limit along each end.
Moreover, the energy function on $\cN_{\Sigma_ { \mbf p}}(X, P)$ takes values in a discrete set
\[
\{  \< [\omega -\mu], B\>  |  B \in H_2^G(X, \Z)\}.
\]

Fix an equivariant homology class $B \in H^G_2(X, \Z)$ such that $\< [\omega -\mu], B\> >0$. Let 
$\cN_{\Sigma}(X, P, B)$ be the moduli space of symplectic vortices on $\Sigma$ associated to
$(P, X)$ with the homology class $B$. We remark that the homology class of $(u, A)$ is defined by the associated orbifold model as in the previous section for one cylindrical end cases. 

Then the Fredholm analysis for a one cylindrical end case in the previous section can be adapted to establish the following theorem.

  \begin{theorem}\label{APS:index-2}Let $ \cN_{\Sigma} (X, P, B; \{(g_i)\}_{i=1, \cdots, k})$ be the subset of
   $\cN_{\Sigma} (X, P, B)$ consisting of 
symplectic vortices  $[(A, u)] $ such that 
  \[
  \p_\infty  (A, u) \in  \big( \cX_0^{(g_1)}  \times \cdots \times\cX_0^{(g_k)}  \big)  \subset (\Cr)^k
     \]
Then  $\cN_{\Sigma} (X, P, B; \{g_i\}_{i=1, \cdots, k})$  admits an orbifold  Fredhom system with its virtual dimension given by
\[
2\<c_1^G(TX), B\> + 2(n-\dim G) (1-g_{\Sigma}) - 2\sum_{i=1}^k  \iota_{CR}( \cX_0^{(g_i)},  \cX_0)
\]
where $g_{\Sigma}$ is the genus of the Riemann surface $\Sigma$. 
 \end{theorem}

\section{Compactness of $L^2$-moduli space of   symplectic vortices}

In this section, we establish a compactness result for the underlying topological space of the  moduli space $ \cN_{\Sigma } (X, P,B)$ of symplectic  vortices on a  Riemann surface $\Sigma$ with $k$ cylindrical ends.    We assume that $k>0$. By reversing the orientation on $S^1$ if necessarily, we can assume  that all these ends are modelled 
on $S^1\times [0, \infty)$.  

Given an orbifold   topological space $\cN$, the underlying topological space (also called the coarse space of $\cN$) will be denoted by $|\cN|$. We will provide a compactification  of the coarse  moduli space $| \cN_{\Sigma } (X, P, B)|$ by adding certain limiting data consisting of bubbling off $J$-holomorphic spheres in  $(X, \omega, J)$ as in the Gromov-Witten theory and  bubbled chains of symplectic vortices on cylinders.   
When $X$ is K\"ahler, the compactness theorem for  the  $L^2$-moduli spaces of  symplectic vortices on  a Riemann surface with  cylindrical end  have been studied in \cite{SV}.

 To describe the limiting data for a sequence of symplectic vortices on  $\Sigma$, we introduce an  index set  for  the topological type of the domain. Let $g$ be the topological  genus of $\Sigma$ and $B\in H^G_2(X, \Z)$ such that   $\< [\omega -\mu], B\> > 0$.  Recall that a tree is a connected graph without any closed cycle of edges.
 

 \begin{definition}\label{web:tree}
 A web  of   stable  weighted  trees  of the  type $(\Sigma;  B)$  is a finite  disjoint union of trees 
 \[
\Gamma =   \Gamma_0 \sqcup   \Gamma_1 \sqcup  \cdots \sqcup \Gamma_k
  \]
consisting of  a  principal tree  $\Gamma_0$  with ordered  $k$  tails and a collection of chains (ordered sequences) of trees 
\[
\Gamma_i  = \bigsqcup_{j=1}^{m_i}  T_i(j) 
\]
for each tail $i =1, \cdots, k$,  
together with  the following additional conditions. 
   \begin{enumerate}
\item  The principal tree $\Gamma_0$ has a distinguished vertex (called the principal root) with  a weight $(g, B_0)$ and  ordered  $k$  tails labelled by $\{1, 2, \cdots,  k\}$.   Here $B_0 \in H_2^G(X, Z)$ satisfies the positivity condition
\[
\< [\omega -\mu], B_0\> \geq 0.
\] 

\item
 For the $i$-th tail in $\Gamma_0$, there is  a chain of  trees  of length $m_i$
 \[
 \Gamma_i = T_i(1) \sqcup T_i (2) \sqcup \cdots \sqcup T_i (m_i)
 \] 
 such that, if non-empty, 
each $ T_i(j)$ has  a distinguished vertex (called a branch  root)  with a weight given  by a class 
$B_{i, j} \in H_2^G(X, Z)$ such that
\[
\< [\omega -\mu], B_{i, j}\> \geq 0.
\]
If $B_{i, j} =0$,  the tree $ T_i(j)$ is non-trivial in the sense that the branch root is not the only vertex. 

\item  Any undistinguished vertex  $v$ in   $\Gamma$    has its weight given  by a class   $B_v\in H_2(X, \Z)$. 
such that
\[
\< [\omega -\mu], B_v\> \geq 0. 
\]
If $B_v=0$, the number of edges  at $v$ is at least 3, two of which have non-zero weights. 

\item Under the natural  homomorphism   $H_2(X, \Z) \to H^G_2(X, \Z)$, 
\ba\label{sum:true}
B_0 + \sum_{i=1}^k\sum_{j=1}^{m_i}    B_{i, j}  + \sum_{i=0}^m \sum_{v\in V(\Gamma_i)}  B_v = B.
\na
Here   $V(\Gamma_i)$ is  the set of undistinguished vertices in $\Gamma_i$. 
\end{enumerate} 
The equivalence between two webs of stable weighted trees can be defined in a usual sense.
 Denote by $\cS_{\Sigma;    B}$ be the  set of equivalence classes of  webs  of   stable weighted   trees of the type $(\Sigma;  B)$. 
  \end{definition}
 
Given two element $[\Gamma] $ and $[\Gamma']$ in $ \cS_{\Sigma;   B}$, we say $[\Gamma] \prec[ \Gamma']$ if
any  representative $\Gamma'$  in $[\Gamma']  $  can be obtained, up to equivalence,  from any representative $\Gamma$  in $[\Gamma]  $   by performing   finitely many steps of the following three
 operations.
\begin{enumerate}
\item Contracting an edge connecting  two  undistinguished  vertices, say $v_1$ and $v_2$,  in 
a tree to obtain a new vertex with  a new  weight $B_{v_1}+ B_{v_2}$.
\item Identifying two branch  roots of   adjacent  trees in a chain   $\Gamma_i$  to get a chain of trees of length $m_i-1$ with  a  weight given by the sum of the two assigned weights.
\item Identifying the principal  root  in $\Gamma$ with a    first branch root 
 in a chain, say  $\Gamma_i$,   such that the new main root  is endowed with a new  weight 
$B_0 + B_{i, 1}$ and the $i$-th chain becomes $T_i (2) \sqcup \cdots \sqcup T_i (m_i)$. 
\end{enumerate}

 \begin{lemma}  $(\cS_{\Sigma;   B}, \prec)$ is a  partially ordered finite  set.
 \end{lemma}
 \begin{proof} It is easy to see that the order $\prec$ is  a partial order. By the condition (\ref{sum:true}), we see that there are only finitely many collections of
 \[
 \{(B_0, B_{i, j}, B_v)| \< [\omega -\mu], B_{i, j}\> >  0,   \< [\omega -\mu], B_v\> >  0\}.
 \]
 The stability conditions for branch roots or undistinguished vertices  with zero weight implies that there
 are only finitely many possibilities. This ensures that 
 $\cS_{\Sigma; B}$ is a finite set. 
 \end{proof}

  Given an element  $\Gamma = \sqcup_{i=0}^k \Gamma_i$ in $\cS_{\Sigma; B}$, we can associate a bubbled Riemann surface  of genus $g$ and $k$ cylindrical ends,   and a collection of  chains of bubbled cylinders  as follows. Associated to $\Gamma_0$,  we assign   a bubbled Riemann surface  $ \Sigma_0 $  which
   is the nodal Riemann surface obtained by attaching  trees of $\C\P^1$'s   to  $\Sigma$.  Associated to  
   an $i$-th chain of trees
  $  \Gamma_i = \sqcup_{j=1}^{m_i} T_i(j)$
 we assign  a chain of  bubbled cylinders
 \[
 C_i = \{C_i(1), \cdots, C_i (m_i)\}
 \]
 where each $C_i(j)$   is a nodal cylinder  with   trees of $\C\P^1$'s attached according to the tree .
 
 Now we construct a moduli space of  stable  symplectic vortices with the domain curve being the bubbled Riemann surface  $\Sigma_0$ or one of  the bubbled cylinders in 
$\{C_i(j)| i=1, \cdots, k; j =1, \cdots, m_i\}$ as follows.

 Let $\tilde \Gamma_0$ be the new weighted 
 graph obtained by severing all edges in $\Gamma_0$ which are attached to the root.  
 Assume  that   $\Gamma_0$ has   $l_0$ tress attached to the root. Then $\tilde \Gamma_0$  consists a single vertex (the root) with $l_0$ half-edges and $k$ ordered tails. The remaining part of $\Gamma_0$, denoted  
 by $\hat\Gamma_0$, becomes  a disjoint union of $l_0$ trees, each of which has a half-edge attached one particular vertex (the  adjacent vertex to the root). Equivalently, 
 \[
 \Gamma_0 = \big( \tilde \Gamma_0 \sqcup \hat \Gamma_0 \big)/\sim
\]
 where the equivalence relation is given by the identification of $l_0$-tuple  half-edges in  $\tilde \Gamma_0$ 
 with  the   $l_0$-tuple  half-edges  in  $\hat \Gamma_0$.

 Denote by $\cN_{\tilde \Gamma_0}$ by the moduli space of 
 symplectic vortices of  homology class $B_0$ over
 $\Sigma$ with  $l_0$  marked points and $k_0$ cylindrical ends. 
 Then  there is a continuous   map
\[
\widetilde{ev}_0:  \cN_{\tilde \Gamma_0} \longrightarrow    X^{l_0} 
\]
given by the evaluations at  the $l_0$  marked points.   Moreover, 
there is a continuous asymptotic limit map along each of the $k$ cylindrical ends
\[
\p_0 :  \cN_{\tilde \Gamma_0} \longrightarrow  (\Cr)^{k}.
\]

Associated to   $\hat \Gamma_0$, as a  disjoint union of $l_0$ trees, there is a moduli space of
the Gromov-Witten moduli space of unparametrized stable  pseudo-holomorphic spheres  with $l_0$-marked points and  the  weighted  dual graph given by $\hat \Gamma_0$, see Chapter 5 in \cite{McS1}. We denote this moduli space by $\cM^{GW}_{\hat \Gamma_0}$.
Then there is a  continuous   map
\[
\widehat{ev}_0:  \cM^{GW}_{\hat \Gamma_0} \longrightarrow  X^{l_0}
\]
given by the evaluations at  the $l_0$  marked points.
 The moduli space of  bubbled symplectic vortices of type $\Gamma_0$, 
 denoted by $\cN_{\Gamma_0}$, is defined to be  the orbifold topological space {\em generated}  by  the fiber product
\[
  \cN_{\tilde \Gamma_0} \times_{ X^{l_0} }   \cM^{GW}_{\hat \Gamma_0}
\] 
with respect to  the maps $\widetilde{ev}_0$ and $\widehat{ev}_0$. 
Then  the coarse moduli space  $|\cN_{\Gamma_0}|$  inherits  a continuous  asymptotic limit map
\[
\p_{\Gamma_0}:  |\cN_{\Gamma_0} | \longrightarrow  (\Cr)^{k}.
\]

\begin{remark}
We remark that there is an ambiguity  here   with regarding the orbifold structure on   $\cN_{\Gamma_0}$. A proper way  to make this precise is to employ the  language  of proper \'etale groupoids  to describe the spaces of  objects and arrows on $\cN_{\tilde \Gamma_0} \times_{ X^{l_0} }   \cM^{GW}_{\hat \Gamma_0}$, and then add further arrows to include 
all equivalences relations  to get an orbifold structure on $\cN_{\Gamma_0}$. As we are dealing with the compactification 
of the coarse moduli space, there is no ambiguity for the coarse space $|\cN_{\Gamma_0}|$. We will return to this issue when we discuss weak Freholm systems for these moduli spaces in \cite{CW2} and \cite{CW3}. 
 \end{remark}

Similarly, for the $i$-th  chain of trees $ \Gamma_i =\bigsqcup_{j=1}^{m_i}  T_i(j)$ , we define a chain of  moduli spaces  of stable  symplectic vortices of type $\Gamma_i$ as follows.  Associated to the  tree $T_i(j)$, we excise the branch root away to get a graph consisting of a single vertex with $l_{i, j}$ half-edges and $l_{i, j}$  trees with one half-edge for each tree.  Let  $\tilde  T_i(j)$ and $\hat T_i(j)$  be these two graphs respectively.  
Denote by $\cN_{i, j}$ be the moduli space of symplectic vortices 
of homology class $B_{i, j}$ over the  cylinder $C_i(j)\cong S^1\times\R$ with $l_{i, j}$-marked points. Then there are a continuous evaluation map
\[
\widetilde {ev}_{i, j}:  \cN_{i, j} \longrightarrow   X^{l_{i, j}}
\]
and   continuous asymptotic value maps
\[
\p^\pm_{i, j}:  \cN_{i, j} \longrightarrow  \Cr
\]
   associated to the two ends at $\pm \infty$ respectively.  Denote  by $\cM^{GW}_{i, j}$ the 
    Gromov-Witten moduli space of unparametrized stable  pseudo-holomorphic spheres  with $l_0$-marked points and  the  weighted  dual graph given by $\hat \Gamma_0$.  Note that $\cM^{GW}_{i, j}$ is equipped with a 
 continuous evaluation map
\[
\widetilde {ev}_{i, j}:  \cM^{GW}_{i, j} \longrightarrow  X^{l_{i, j}}.
\]
Then  by adding all arrows to  the fiber product 
\[
    \cN_{i, j} \times_{ X^{l_{i, j}} } \cM^{GW}_{i, j}, 
\]
we get   the moduli space $\widehat \cN_{T_i(j)}$ of  stable  symplectic vortices of type $T_i(j)$.
In particular,  $ |\widehat  \cN_{T_i(j)} |$ inherits   continuous  asymptotic limit maps
\ba\label{bdy:i-j}
\hat \p^{\pm}_{T_i(j)}:  |\widehat  \cN_{T_i(j)} |  \longrightarrow   \Cr 
\na
along the two ends. Note that the group of  rotations and translations   $S^1\times \R$ on the cylinder  induces a free action of $S^1\times \R$ on the  moduli space $ \widehat  \cN_{T_i(j)}$ which  preserves the asymptotic limit maps $\hat \p^{\pm}_{T_i(j)}$ invar. We quotient the  moduli space $ \widehat  \cN_{T_i(j)}$ by the  group $\R \times S^1$, and denote the resulting
moduli space by
\[
 \cN_{T_i(j)} =  \widehat  \cN_{T_i(j)} /(\R\times S^1).
 \]
The induced   asymptotic limit maps on the coarse moduli space is denoted  by
 \[
 \p^{\pm}_{T_i(j)}:   | \cN_{T_i(j)} |  \longrightarrow   \Cr.
 \]

 By taking the consecutive  fiber products with respect to maps $\p^+_{T_i(j)}$ and 
 $\p^-_{T_i(j+1)}$ for $j=1, \cdots, m_i$,  we get the  coarse moduli spaces of chains of stable  symplectic vortices of type $\Gamma_i$, that is,
\[
|\cN_{\Gamma_i} | = | \cN_{T_i(1)} |  \times_{\Cr}  |\cN_{T_i(2)} |  \times _{\Cr}  \cdots  \times_{\Cr}  |\cN_{T_i(m_i)}|.
\]
Then there are two asymptotic limit maps given by $\p^-_{T_i(1)} $ and  $\p^+_{T_i(m_i)}$, simply denoted 
by  $\p^-_i $ and  $\p^+_i$,
\[
\p^\pm_i:   |\cN_{\Gamma_i}|  \longrightarrow   \Cr. 
\]

\begin{definition} Given $\Gamma   = \Gamma_0  \sqcup   \Gamma_1 \sqcup  \cdots \sqcup \Gamma_k$ a  web of stable weighted trees  in $\cS_{\Sigma; B}$,   the  coarse moduli space of stable symplectic vortices of type $\Gamma$, denoted by $|\cN_{\Gamma}|$,  is defined to be the fiber product
\[
|\cN_{\Gamma}| = | \cN_{\Gamma_0}| \times_{(\Cr)^k} \prod_{i=1}^k | \cN_{\Gamma_i}|,
 \]
 where $ \prod_{i=1}^k  |\cN_{\Gamma_i} |= |\cN_{\Gamma_1}|  \times  |\cN_{\Gamma_2} | \times \cdots \times 
 | \cN_{\Gamma_k}|$, and the fiber product is defined via the maps $\p_{\Gamma_0}:  |\cN_{\Gamma_0}| \to (\Cr)^k$
  and   
\[ \prod_{i=1}^k  \p^-_i:   \prod_{i=1}^k  |\cN_{\Gamma_i}| \to (\Cr)^k.\]
  There exists a continuous map
\[
\p_\Gamma:  |\cN_{\Gamma}| \longrightarrow  (\Cr)^{k}
\]
given by $ \prod_{i=1}^k  \p^+_i$.

   \end{definition}

   For any $k$-tuple $((g_1), \cdots, (g_k))$ conjugacy classes in $G$  such that each representative $g_i$ in $(g_i)$
    has a non-empty fixed point set in $\mu^{-1}(0)$, then we define 
\[
| \cN_{\Gamma} ((g_1), \cdots, (g_k)) | = \p_\Gamma^{-1} \left( |\cX_0^{(g_1)}| \times \dots \times | \cX_0^{(g_k)}|
 \right). 
\]   
Now we can state the compactness theorem for the coarse  $L^2$-moduli space 
$|\cN_{\Sigma}(X, P, B)|$ of symplectic vortices on $\Sigma$.

\begin{theorem}  \label{thm:cpt}
Let $\Sigma$ be a Riemann surface of genus $g$  with $k$-cylindrical ends. The  coarse $L^2$-moduli space 
$|\cN_{\Sigma}(X, P, B)|$ can be compactified to    a  stratified  topological space 
\[
|\overline{\cN}_{\Sigma}(X, P, B) | = \bigsqcup_{\Gamma \in \cS_{\Sigma;  B}} | \cN_\Gamma |
\]
such that the top stratum is $|\cN_{\Sigma}(X, P, B)| $. 
Moreover, the  coarse moduli space $$ |\cN_{\Sigma} (X, P, B; \{(g_i)\}_{i=1, \cdots, k})|$$  of the
moduli space $ \cN_{\Sigma } (X, P, B; \{(g_i)\}_{i=1, \cdots, k})$  in Theorem \ref{APS:index-2}
can be  compactified to a stratified topological    space 
\[
|\overline{\cN}_{\Sigma}(X, P, B; \{(g_i)\}_{i=1, \cdots, k})|  = \bigsqcup_{\Gamma \in \cS_{\Sigma;  B}} | \cN_\Gamma  ((g_1), \cdots, (g_k))|.
\]
\end{theorem}

\begin{proof}  For simplicity, we assume that   the Riemann surface $\Sigma$ has only one  outgoing cylindrical end, that is, diffeomorphic to  $S^1\times [0, \infty)$.  The proof for the general case is essentially the same.
Under this assumption,   any web of stable weighted trees of the type  $(\Sigma; B)$ has  only one chain of trees denoted by
$
\{T(1), T(2), \cdots, T(m)\}.
$

Given any sequence $[A_i, u_i] \in \cN_{\Sigma}(X, P, B)$, we shall show that there exists a subsequence with a limiting datum in $ \cN_\Gamma$ for some $\Gamma \in  \cS_{\Sigma;   B}$.  The strategy to this claim  is quite standard now, for example see \cite{Fl}, \cite{Don1} and \cite{MT}. 

Note that the energy function on this sequence 
\[
E(A_i, u_i) = \int_{\Sigma} \dfrac 12 (|d_{A_i}u_i|^2 + |F_{A_i}|^2 + |\mu\circ u_i|^2) \nu_\Sigma  
\]
is constant given by  $\< [\omega -\mu], B\>.$  For any non-constant pseudo-holomorphic map from a closed Riemann surface, the energy is bounded from below by a positive number 
\[
\min \{ \<[\omega], \beta\> | \beta \in H_2(X, \Z),  \<[\omega], \beta\> >0\}, 
\]
which is greater than the minimal energy of non-trivial symplectic vortices on $\Sigma$ associated to
$(P, X, \omega)$
\[
 \hbar =\min \{ \<[\omega-\mu], \beta\> | \beta \in H^G_2(X, \Z),  \<[\omega-\mu], \beta\> >0\}, 
\]

\noindent {\bf Step 1.}  (Convergence for the sequence with bounded  derivative) Without loss of generality,  we suppose that 
$\{(A_i, u_i)\}$  is a sequence of    symplectic vortices   in $\cN_{\Sigma}(X, P, B)$ with a uniform bound
\[
 \| d_{A_i} u_i\|_{L^\infty} < C
\]
for a constant $C$. Then there exists a sequence of  gauge transformations $\{ g_i\}$ such that 
$\{g_i\cdot (A_i, u_i)\}$ has a $C^\infty$ convergent subsequence.

This claims  follows from Theorem 3.2 in \cite{CGRS}.  

\noindent {\bf Step 2.}  (Bubbling phenomenon at interior points) Assume that the sequence $ \| d_{A_i} u_i\|_{L^\infty}$ is unbounded over a compact set in $\Sigma$, then the rescaling technics in the  proof of Theorem 3.4 in \cite{CGRS} can be applied here to get the standard pseudo-holomorphic sphere at the point in $\Sigma$ where a sphere is attached to $\Sigma$.  

Hence, combining Steps 1-2, we know that there may exist a subset of finite points, say $\{q_1, \cdots, q_{l_0}\}$, of
$\Sigma$ such that for any compact set $Z\subset  \Sigma' = \Sigma - \{q_1, \cdots, q_{l_0}\}$, 
there exists a subsequence of $(A_i,u_i)$ and gauge transformation $g_i$ such that $g_i(A_i,u_i)$ uniformly converge in $Z$.  As $Z$ exhausts $\Sigma'$, we get a symplectic
vortex $(A_\infty,u_\infty)$ on $\Sigma'$. By the removable singularity theorem, this   symplectic
vortex $(A_\infty,u_\infty)$  can be defined on $\Sigma$.

Moreover, at each point $q_j$, we get a  bubble tree of holomorphic sphere attached to $q_i$.
As in the Gromov-Witten theory,  there is no energy lost when the bubbling phenomenon happens  at interior points. This gives rise to a principal tree $\Gamma_0$  in a web of stable weighted trees in $\cS_{\Sigma; B}$. 

\noindent {\bf Step 3.}  (Bubbling phenomenon at  the infinite end)  Assuming  that for a sufficiently large $T$, the sequence $\{(A_i, u_i)\}$ converges to $(A_\infty,u_\infty)$ on  $\Sigma - (S^1 \times [T, \infty))$,
where $(A_\infty,u_\infty)$ is of the type $\Gamma_0$, a principal tree in Definition \ref{web:tree}.
 Now we study the sequence over the cylindrical end. We may further assume that  the Yang-Mills-Higgs energy 
 \[
  \int_{S^1\times [T, \infty)} \dfrac 12 (|d_{A_i}u_i|^2 + |F_{A_i}|^2 + |\mu\circ u_i|^2) 
 \]
 is greater than the minimum energy $\hbar$ defined as above. Otherwise, the limit of the sequence is in  $\cN_{\Gamma_0}$. 

We replace the sequence $\{(A_i, u_i)|_{(S^1 \times [T, \infty)}\}$ by their translations  to the left by $\{\delta_i\}$ 
such that the  Yang-Mills-Higgs energy  of the translate  for    $(A_i, u_i)$ over $[T-\delta_i, 0]$ 
is $ \hbar/4.$ Then $\delta_i \to \infty$ as $i\to \infty$.  Applying  the above standard convergence theorem 
to the  translated sequence with and without the  bounded   derivative condition on $ \| d_{A_i} u_i\|_{L^\infty}$, there exists a subsequence which converges to a  bubbled symplectic vertex $(A'_\infty,u'_\infty)$   on any compact subset of   $S^1\times \R$.
This gives rise to a stable symplectic vortex of type $\Gamma (1)$, where $\Gamma (1)$ is a tree with a branch root as in Definition \ref{web:tree}.

\noindent {\bf Step 5.} (No energy loss in between) Now we show that there is no energy loss on the connecting neck between
$(A'_\infty,u'_\infty)$  and $(A_\infty,u_\infty)$. Equivalently, associated to  the subsequence (still denoted by  
$\{(A_i, u_i)\}$,  for each $i$, there exist 
\[
N_i<N_i+K_i<N_i'-K_i<N_i'
\]
such that  $N_i,K_i$ and $N_i'-N_i -2K_i  \to \infty$ as $i\to \infty$, and under the temporal gauge, 
\begin{enumerate}
\item the sequence  $(A_i,u_i)$ on  $S^1 \times [N_i,N_i+K_i]$  coneverges  to
$(A_\infty,u_\infty)$ on any compact set after translation;
\item $(A_i,u_i)$ on   $ S^1 \times [N_i'-K_i,N_i]$   coneverges  to
$(A'_\infty,u'_\infty)$ on any compact set after translation.
\end{enumerate}
 We shall show that 
the  Yang-Mills-Higgs energy of $(A_i,u_i)$ on  $S^1\times [N_i+K_i,N_i'-K_i]$
tends to 0 as $i\to \infty$.

Let $y_\infty$ and $y'_{-\infty}$ be the limit of $(A_\infty,u_\infty)$
as $t\to\infty$
and $(A'_\infty,u'_{\infty})$ as $t\to -\infty$ respectively. Let $\bar y'_{-\infty}$ be the pair obtained from $y'_{-\infty}$ by reversing the orientation of $S^1$. Suppose that
$$
y_\infty=(a, \alpha),\;\;\;\bar y'_{-\infty}=(b,\beta).
$$
Then $(A_i(t),u_i(t)), t\in [N_i,N_i+K_i]$ is arbitrary close to $y_\infty$  and
$(A_i(t'),u_i(t')), t'\in [N_i'-K_i,N_i']$ is arbitrary close to $\bar y'_{-\infty}$ as $i\to \infty$.

We   claim  that $\widetilde {\mc L}(y_\infty)=\widetilde {\mc   L}(\bar y'_{-\infty})$.
Otherwise, the difference  would be larger than
$\hbar$. However the Yang-Mills-Higgs energy of $(A_i,u_i)$ on   $[N_i,N_i']$ is less than $\hbar/2$. This is impossible.

Now we explain the Yang-Mills-Higgs  energy of $(A_i,u_i)$ at
$[N_i+t, N_i'-t]$ decays exponentially with respect to $t$.
We normalize the band by translation  such that  $[N_i+K_i,N_i'-K_i]$   becomes
$
[-d,d]$ where $d=\frac{N_i'-N_i}{2} - K_i$.

Denote the Yang-Mills-Higgs energy of $y_i=(A_i,u_i)$ on   $S^1\times [-t,t]$ 
by 
\[
E_i(t) =  \int_{S^1\times [-t,t]} \dfrac 12 (|d_{A_i}u_i|^2 + |F_{A_i}|^2 + |\mu\circ u_i|^2)  
\]
 for  $0\leq t\leq d$. Then
$$
\frac{dE_i(t)}{dt}=\|\nabla \widetilde {\mc L}_{y_i(t)}\|^2+
\|\nabla \widetilde {\mc L}_{y_i(-t)}\|^2
$$
Replace  $\widetilde {\mc L}$ by  $ \widetilde {\mc L}-\widetilde {\mc L}(y_\infty)$. Then by the crucial inequality (Proposition \ref{cru:inequ2}), we obtain the following differential inequality
\ba\label{diff:inequ}
\frac{dE_i(t)}{dt} \geq \delta(|\tilde {\mc L}({y_i(t)})|+
| \tilde {\mc L}({y_i(-t)})|)
\geq \delta E_i(t).
\na
Here we use the fact that
$$
E_i(t)= |\widetilde{\mc L}(y_i(t))-\widetilde{\mc L}(y_i(-t))|.
$$
Then the differential inequality  (\ref{diff:inequ}) 
implies 
$$
e^{-\delta t}E_i(t)\leq e^{-\delta d}E_i(d).
$$
Apply to our case, this implies 
\begin{equation}
E(A_i,u_i)|_{[N_i+K_i,N_i'-K_i]}
\leq e^{-\delta K_i}E(A_i,u_i)|_{[N_i,N_i']}.
\end{equation}
As $K_i\to \infty$, the Yang-Mills-Higgs energy goes to 0.

This ensures that 
\[
\partial_{\Gamma_0} ([A_\infty,u_\infty]) = \partial_{\Gamma_1}^-  ([A'_\infty,u'_\infty])   \in \Cr.
\]
If the sum of Yang-Mills-Higgs energies of  $(A_\infty,u_\infty)$ and $(A'_\infty,u'_\infty)$ agrees with $\<[\omega-\mu], B\>$,  the limit of the sequence   is in $\cN_{\Gamma}$ for $\Gamma = \Gamma_0 \sqcup \Gamma_1$.

\noindent {\bf Step 6.} (Energy loss at the $+\infty$ end in the limit) 
If the sum of Yang-Mills-Higgs energies of  $(A_\infty,u_\infty)$ and $(A'_\infty,u'_\infty)$  is less than $\<[\omega-\mu], B\>$,  then 
\[
\nu = \<[\omega-\mu], B\>  - E(A_\infty,u_\infty) -   E(A'_\infty,u'_\infty) \geq \hbar.
\]
In this case, we  loss some energy at the $+\infty$ end in the limit, we repeat Steps 3-4 to get a limit in $\cN_{\Gamma}$ with  a chain of trees of length $m\geq 2$.  This  same process will  stop after a finitely many steps due to the fact that each tree in the chain carries at least $\hbar$ energy.
This completes the compactification of   $|\cN_{\Sigma}(X, P, B)|$.

The compactification of $|\cN_{\Sigma_{\mbf p}} (X, P, B; \{(g_i)\}_{i=1, \cdots, k})|$ can be obtained in the similar manner.

\end{proof}

\section{Outlook}\label{outlook}

In this paper, we mainly discuss the $L^2$-moduli space of symplectic vortices on a   Riemann surface with cylindrical end. The analysis can be generalised to the case of a family of  Riemann surfaces with cylindrical end. Then we
get a moduli space of $L^2$-symplectic vortices fibered over the moduli space of complex structures. In particular, for a Riemann surface
\[
\Sigma_{g, k} = (\Sigma, (z_1, \cdots, z_k), \fj)
\]
 of genus $g$ and with $k$-marked points, when $2-2g-k< 0$, we can consider $\Sigma_{g, k}$ as a  Riemann surface   of genus $g$ and with $n$-punctures. By the uniformization theorem, for each complex structure on  $\Sigma_{g, k}$, there is a unique  complete hyperbolic metric on  the  corresponding punctured  surface.  This defines a canonical horodisc structure at each puncture, see \cite{CLW2}. This    horodisc structure at each puncture is also called a hyperbolic cusp.
 Using the canonical  horodisc structure at each point defined the complete hyperbolic metric on  the punctured $\Sigma_{g, k}$, we can identify the moduli space $\cM_{g, k}$ with  the moduli space of hyperbolic metrics with a canonical  horodisc structure at each  punctured disc. Each horodisc can be equipped with a canonical cylindrical metric
 on the punctured disc. In particular, we get a smooth universal family of Riemann surface with $k$ cylindrical ends over
 the moduli space $\cM_{g, k}$. Then the analysis in this paper on the $L^2$-moduli space of symplectic vortices can be carried over to get a continuous family of Fredholm system defined by the  symplectic vortex equations. 
  The corresponding  $L^2$-moduli spaces of  symplectic vortices without and  with  prescribed  asymptotic data  will denoted by
\[
\cN_{g, k}(X, P, B)  \qquad \text{and }\qquad \cN_{g, k}(X, P, B;  \{(g_i)\}_{i=1, \cdots, k})  \]
respectively. 
Then we have the similar  compactness result for this $L^2$-muduli space where the index set $\cS_{\Sigma; B}$ is
replaced by $\cS_{g, k; B}$ where the root of a principal part of each web is replaced  a dual  graph as in the Gromov-Witten moduli space with weights in $H_2^G(X, \Z)$ at each vertex,   each vertex  carries  bubbling trees (with  weights in  $H_2(X, \Z)$) and each tail is assigned a chain of trees. 

In the subsequence paper, we shall  also establish a weak  orbifold Fredholm system and a gluing principle   for the compactified moduli space $\overline{\cN}_{g, k} (X, P, B; \{(g_i)\}_{i=1, \cdots, k})$  so that the virtual neighbourhood technique developed in  \cite{CLW3} can be applied to define a Gromov-Witten type invariant from these moduli spaces(\cite{CW2}).
We will show that the compactified moduli space $\overline{\cN}_{g, k} (X, P, B; \{g_i\}_{i=1, \cdots, k})$
admits an oriented  orbifold virtual system and the virtual integration  
\[
\int^{vir}_{\overline{\cN}_{g, k} (X, P, B; \{(g_i)\}_{i=1, \cdots, k})}:  H^*(I\cX_0, \R) ^k \to \R
\] 
is well-defined. Here $ I\cX_0$ is the inertial orbifold of the symplectic reduction $\cX_0 = \mu^{-1}(0)/G$. 
The Gromov-Witten type invariant  is defined to be 
 \[
 \< \alpha_1, \cdots, \alpha_k\>^{\ell HGW}_{g, k, B} = \int^{vir}_{\overline{\cN}_{g, k} (X, P, B; \{(g_i)\}_{i=1, \cdots, k})}
 \partial_\infty^* (\pi_1^* \alpha_1 \wedge  \cdots  \pi_k^*\wedge \alpha_k )
 \]
 for any  $k$-tuple of cohomology classes  
\[
(\alpha_1, \cdots, \alpha_k) \in H^*(\cX_0^{(g_1)}, \R) \times \cdots \times H^*(\cX_0^{(g_k)}, \R).
\]
Here $\pi_i:  \Cr^k \to  \Cr=I\cX_0$ denotes the projection to the $i$-the component. We emphasize that this is an invariants on $H^\ast_{CR}(\mc X_0)$ rather than 
on $H^\ast_G(X)$. It is different from usual HGW invariants. We call the invariant $L^2$-Hamiltonian GW invariants (abbreviated as $\ell$HGW).
In particular, when $(g, k) = (0, \geq 3)$, the above invariant can be assembled   to get a new (big) quantum product $*^{HR}$ 
on $ H^*(I\cX_0, \R)$.   Here HR stands for Hamiltonian reduction. In a separate paper(\cite{CW3}), we will introduce   an augmented 
symplectic vortex equation to  define an equivariant version of this invariant on $H^*_G(X)$ when $G$ is abelian. This enables us to define a quantum product $\ast_G$ on $H^\ast_G(X)$. We 
 investigate its relation
to $*^{HR}$, in particular, we combine symplectic vortex equation with the augmented one to define the quantum Kirwan map $Q_\kappa$ and show that $Q_\kappa$ is a ring morphism with respect to $\ast_G$ and $\ast^{HR}$.

   \vskip .2in

\noindent  
{\bf Acknowledgments}  This work is partially supported by   the Australian Research Council Grant
(DP130100159) and  the National Natural Science Foundation of China Grant.  The first author would also like to thank the hospitality of BICMR and ANU. We acknowledge that  the compactness theorem  in this manuscript, mainly   overlaps with the result  an earlier paper  \cite{SV}  by Venugopalan. We only noticed her paper  recently while we are  preparing the sequels  \cite{CW2} and \cite{CW3}. 
The approaches we employed  here is very different from those  in \cite{SV}.


\begin{thebibliography}{9999}

\bibitem{AB}  M. F. Atiyah and R. Bott, {\em The Yang-Mills equations over Riemann surfaces.}  Philos. Trans. Roy.
Soc. London Ser. A 308 (1983), 523-615.

\bibitem{ABr} D. Austin and P. Braam, {\em Morse-Bott theory and equivariant  cohomology}. 
The Floer memorial volume, 123-183, Progress in   Math., 133, Birkh\"auser, Basel, 1995. 


\bibitem{BDW}  S. B. Bradlow, G. Daskalopoulos, and R. Wentworth, {\em Birational equivalences of vortex moduli.}
Topology 35 (1996), 731-748.


\bibitem{CKM}  I. Ciocan-Fontanine, B. Kim, and D. Maulik,
{\em Stable quasimaps to GIT quotients},  arXiv:1106.3724. 


\bibitem{CLW1} B. Chen, A. Li and B. Wang, {\em 
	Virtual neighborhood technique for pseudo-holomorphic spheres},
 arXiv:1306.32
 
\bibitem{CLW2} B. Chen, A. Li and B. Wang, Gluing principle for orbifold stratified spaces. Preprint.


\bibitem{CLW3} B. Chen, A. Li and B. Wang, {\em 
 Smooth structures on moduli spaces of stable maps,} in preparation.
 
 \bibitem{CW2}  B. Chen and B. Wang, {\em  $L^2$-symplectic votices and  
 Hamiltonian Gromov-Witten invariants},   in preparation.
 
 \bibitem{CW3} B.  Chen and B. Wang, {\em    Augmented symplectic vortices,
 Hamiltonian equivariant Gromov-Witten invariants and quantization  of Kirwan morphisms}, in preparation.

 
\bibitem{CR} W. Chen, Y. Ruan, {\em A new cohomology theory of orbifold,}  Comm. Math. Phys. {\bf  248}  (2004), no.1,  1-31.



\bibitem{CGS} K.Cieliebak, A.Gaio and D. Salamon, {\em J-holomorphic curves, moment maps,
and invariants of Hamiltonian group actions,}  IMRN 10 (2000), 831-882.

\bibitem{CGRS} K. Cieliebak, A.R. Gaio, I. Mundet i Riera, D.A. Salamon, {\em The symplectic vortex equations and invariants of Hamiltonian group actions}.  J. Symplectic Geom. 1 (3) (2002) 543-645.


\bibitem{Don1} S. K. Donaldson, {\em Floer homology groups in Yang-Mills theory,}  vol. 147 of Cambridge Tracts in Mathematics. Cambridge University Press (2002). With the assistance of M. Furuta and D. Kotschick.


\bibitem{Don}  S. K. Donaldson, {\em Moment maps and diffeomorphisms.} Asian J. Math. 3 (1999), 1-15.
\bibitem{DK} S. Donaldson, P. Kronheimer, {\em The Geometry of Four-Manifolds.}
  Oxford Math. Monogr., Oxford University Press, 1990.

\bibitem{Fl} A. Floer, {\em Symplectic fixed points and holomorphic spheres.}  Comm. Math. Phys. 120 (4) (1989) 575-611.

 \bibitem{FHS} A. Floer,  H.Hofer, D. Salamon,  {\em Transversality in elliptic Morse theory for
the symplectic action.}  Duke Math. Journal 80 (1996), 251-292.

\bibitem{Fr1} U. Frauenfelder, {\em The Arnold-Givental conjecture and moment Floer homology, }Int.Math.Res.Not. 2004, no. 42, 2179-2269.

\bibitem{Fr2}  U. Frauenfelder, {\em Vortices on the cylinder,}   Int.Math.Res.Not. 2006, Art. ID 63130.


  \bibitem{FO99}     K. Fukaya and  K. Ono, {\em Arnold conjecture and Gromov-Witten invariant.}  Topology 38 (1999),
no. 5, 933-1048.

      \bibitem{FOOO}  K. Fukaya, Y.-G. Oh, H. Ohta, and K. Ono, {\em Lagrangian Intersection Theory, Anomaly and
Obstruction, Parts I and II. }  AMS/IP Studies in Advanced Mathematics, Amer. Math. Soc.
and Internat. Press.

\bibitem{GaS} A. Gaio, D. Salamon, {\em Gromov-Witten invariants of symplectic quotients and adiabatic limits.} Journal of Symplectic Geometry 3 (1) (2005) 55-159.



\bibitem{Gro}  M. Gromov,  {\em Pseudo-holomorphic curves in symplectic manifolds.}  Invent. Math. Vol. 82 (1985), 307-347.

\bibitem{GS}  V. Guillemin,  S.  Sternberg, {\em Symplectic techniques in physics.}  Second edition. Cambridge University Press,  1990. 
 
 \bibitem{JaT} A. Jaffe, C. Taubes, {\em Vortices and monopoles: structure of static gauge theories.} 
  Progress in Physics, Vol. 2, Birkhauser, 1980.

 \bibitem{LiRuan}   A. Li, Y. Ruan, {\em Symplectic surgery and Gromov-Witten invariants
      of Calabi-Yau 3-folds.}  Invent. Math. 145(2001), 151-218.



 \bibitem{LiuTian98}  G. Liu and    G. Tian,  {\em  Floer homology and Arnold conjecture.}  J. Differential Geom. 49 (1998),
no. 1, 1-74.

\bibitem{MMR}  J.  Morgan, T. Mrowka, D. Ruberman,   {\em The $L^2$-moduli space and a vanishing theorem for Donaldson polynomial invariants. }
Monographs in Ge- ometry and Topology, II. International Press (1994).

 \bibitem{McS0} D. McDuff and D. Salamon, {\em J-holomorphic curves and quantum cohomology.}
  University Lecture Series, 6. American Mathematical Society, Providence, RI, 1994.
  
  \bibitem{McS1} D.  McDuff and   D. Salamon, {\em $J$-holomorphic curves and Symplectic Topology}, AMS Colloquium Publications, Vol. 52.

\bibitem{MR}   I. Mundet i Riera, {\em Hamiltonian Gromov-Witten invariants}, Topology 42 (2003) 525-553.

\bibitem{Mundet}I. Mundet i Riera, {\em A Hitchin-Kobayashi correspondence for K\"ahler fibrations.} J. Reine Angew. Math.,
528(2000), 41-80.

\bibitem {MT}  I. Mundet i Riera, G. Tian, {\em A compactification of the moduli space of twisted holomorphic maps.} Advances in Mathematics 222 (2009) 1117-1196.

\bibitem {MT_HGW}  I. Mundet i Riera, G. Tian, private communication.
\bibitem{Woo} C. Woodward, {\em Quantum Kirwan morphism and Gromov-Witten invariants of quotients.} 
arXiv:1204.1765.
 
 
 \bibitem{Ott}  A. Ott, {\em Removal of singularities and Gromov compactness for symplectic vortices}, arXiv:0912.2500. 

\bibitem{SL}  R. Sjamaar and  E. Lerman, {\em  Stratified symplectic spaces and reduction. }
Annals of Mathematics, 1991, 134, 375-422.

\bibitem{Tau}  C. Taubes, {\em $L^2$-moduli spaces on 4-manifolds with cylindrical ends. }
Mono- graphs in Geometry and Topology, I. International Press (1993).

\bibitem{SV}  S. Venugopalan, {\em Vortice on surfaces with cylindrical ends,}  arXiv:1312.1074.


\bibitem{Zil} F. Ziltener, {\em A quantum Kirwan map: bubbling and Fredholm theory.}  Memiors of the American Mathematical Society, 2012.


\bibitem{Zil1} F.  Ziltener,  {\em The invariant symplectic action and decay for vortices,}  J. Symplectic
Geom. 7 (2009), no. 3, 357?376.


\end{thebibliography}
\end{document}